\documentclass[english]{amsart}

\usepackage{esint}
\usepackage[svgnames]{xcolor} 
\usepackage[colorlinks,citecolor=red,pagebackref,hypertexnames=false,breaklinks]{hyperref}
\usepackage{pgf,tikz}
\usepackage{pdfsync}

\usepackage{dsfont}
\usepackage{url}
\usepackage[utf8]{inputenc}
\usepackage[T1]{fontenc}
\usepackage{lmodern}
\usepackage{babel}
\usepackage{mathtools}  
\usepackage{amssymb}
\usepackage{lipsum}
\usepackage{mathrsfs}
\usepackage{color}
\usepackage{stmaryrd}

\newtheorem{theorem}{Theorem}[section]
\newtheorem{proposition}{Proposition}[section]
\newtheorem{lemma}{Lemma}[section]

\newtheorem{corollary}{Corollary}[section]

\newtheorem{remark}{Remark}[section]

\numberwithin{equation}{section}

\title[Two inverse parabolic problems]{Two parabolic inverse problems for an equation with unbounded zero-order coefficient}

\author[Mourad Choulli]{Mourad Choulli}
\address{Universit\'e de Lorraine}
\email{mourad.choulli@univ-lorraine.fr}

\begin{document}

\begin{abstract}
This work is composed of two parts. We prove in the first part the uniqueness of the determination of the unbounded zero-order coefficient in a parabolic equation from boundary measurements. The novelty of our result is that it covers the largest class of unbounded zero-order coefficients. We establish in the second part a logarithmic stability inequality for the problem of determining the initial condition from a single interior measurement. As by-product, we  obtain an observability inequality for a parabolic equation with unbounded  zero-order coefficient. The proof of this observability inequality is based on a new global quantitative unique continuation for the Schr\"odinger equation with unbounded potential.  For the sake of completeness, we provide in Appendix \ref{appA} a full proof of this result.
\end{abstract}

\subjclass[2010]{35R30}

\keywords{Parabolic inverse problems, unbounded zero-order coefficient, boundary measurements.}

\maketitle


\section*{First part: Determination of zero-order coefficient from boundary measurements}

\section{Introduction}

Let $\Omega$ be a bounded domain of $\mathbb{R}^n$ of class  $C^{1,1}$, $n\ge 5$,  with boundary $\Gamma$, and  $q=2n/(n+4)$. Since we are dealing with parabolic equations, all functions are assumed to have real values. Throughout this part, $0<t_0<t_\ast <\mathfrak{t}$ and $0<\epsilon <t_\ast-t_0$. We fix $0<\alpha \le 1$ and we equip 
\[
\mathcal{G}=\{G\in W^{2,1}((0,\mathfrak{t}),L^2(\Omega))\cap C^\alpha([0,\mathfrak{t}],H^2(\Omega));\; \mathrm{supp}\; G\subset \overline{\Omega}\times [t_0,\mathfrak{t}]\}
\]
 with its natural norm
\[
\|G\|_{\mathcal{G}}:=\|G\|_{W^{2,1}((0,\mathfrak{t}),L^2(\Omega))}+\|G\|_{C^\alpha([0,\mathfrak{t}],H^2(\Omega))}.
\]
Define
\[
\mathcal{B}=\{\varphi=G_{|\Sigma_{\mathfrak{t}}};\; G\in \mathcal{G}\},
\]
where $\Sigma_{\mathfrak{t}}:=\Gamma \times (0,\mathfrak{t})$ and
\[
\dot{\varphi}:=\{ G\in \mathcal{G}; G_{|\Sigma_{\mathfrak{t}}}=\varphi\},\quad \varphi \in \mathcal{B}.
\]
 $\mathcal{B}$ will be endowed with the quotient norm
\[
\|\varphi\|_{\mathcal{B}}:=\inf\{\|G\|_{\mathcal{G}};\; G\in \dot{\varphi}\},\quad \varphi \in \mathcal{B}.
\]
Let $\gamma_0$  and $\gamma_1$ denote the bounded trace operators from $W^{2,q}(\Omega)$ onto $W^{2-1/q,q}(\Gamma)$  and from $W^{2,q}(\Omega)$ onto $W^{1-1/q,q}(\Gamma)$) given by
\[
\gamma_0w=w_{|\Gamma},\quad \mathrm{and}\quad  \gamma_1w:=\partial_\nu w,\quad w\in C^\infty(\overline{\Omega}),
\]
respectively, where $\partial_\nu$ represents the derivative along the unitary exterior normal vector field $\nu$. When $w\in C([0,\mathfrak{t}],W^{2,q}(\Omega))$, we use the notations $\gamma_0w$ and $\gamma_1w$ for $t\mapsto \gamma_0w(\cdot ,t)$ and $t\mapsto \gamma_1w(\cdot,t)$, respectively.

The notation $Q_{\mathfrak{t}}:=\Omega\times (0,\mathfrak{t})$  will be used hereinafter. We prove in the next section (Proposition \ref{pro2}) that, for all $V\in L^{n/2}(\Omega)$ and $\varphi \in \mathcal{B}$, the IBVP
\begin{equation}\label{p1}
\left\{
\begin{array}{ll}
(\partial_t -\Delta +V)u=0\quad &\mathrm{in}\; Q_{\mathfrak{t}},
\\
\gamma_0u=\varphi &\mathrm{on}\; \Sigma_{\mathfrak{t}} ,
\\
u(\cdot,0)=0
\end{array}
\right.
\end{equation}
admits a unique solution $u_V(\varphi)\in C^1([0,\mathfrak{t}],L^2(\Omega))\cap C([0,\mathfrak{t}],W^{2,q}(\Omega))$.

Let $\Gamma_0$ and $\Gamma_1$ be two nonempty open subsets of $\Gamma$ satisfying $\Gamma_1\cap \Gamma_0\ne \emptyset$. We define $\mathcal{F}_0$ and $\mathcal{B}_0$ in the same way as $\mathcal{F}$ and $\mathcal{B}$ by replacing in the definition of $\mathcal{F }$ $\mathrm{supp}\; G\subset \overline{\Omega}\times [t_0,\mathfrak{t}]$ by $\mathrm{supp}\; G\subset (\Omega\cup \Gamma_0)\times [t_0,t_\ast-\epsilon]$. In this case, $\varphi\in \mathcal{B}_0$ satisfies $\mathrm{supp}\; \varphi \subset \Gamma_0\times [t_0,t_\ast-\epsilon]$.

\begin{theorem}\label{thm1}
Assume that $\Gamma_0\cup \Gamma_1=\Gamma$. If $V_1,V_2\in L^{n/2}(\Omega)$ satisfy
\[
\gamma_1u_{V_1}(\varphi)(\cdot,t_\ast)_{|\Gamma_1}=\gamma_1u_{V_2}(\varphi)(\cdot,t_\ast)_{|\Gamma_1},\quad  \varphi\in \mathcal{B}_0,
\]
then $V_1=V_2$.
\end{theorem}

Theorem \ref{thm1} extends a result of \cite{CK}. The authors claim in \cite{CK} that this result is true when $V\in L^{r}(\Omega)$ with $r>n/2$. However, upon careful examination of their proof, one sees that they need $Vu\in L^2(\Omega)$ if $u\in H^1(\Omega)$. Since $H^1(\Omega)$ is continuously embedded in $L^{2n/(n-2)}(\Omega)$, $Vu\in L^2(\Omega)$ only if $V\in L^{r}(\Omega)$ with $r\ge n$.

Without assuming that $\Gamma_0\cup \Gamma_1=\Gamma$, we have the following result.

\begin{theorem}\label{thm2}
Let $\Omega_0$ be a neighborhood of $\Gamma$ in $\overline{\Omega}$ and $V_1\in L^{n/2}(\Omega)$ satisfying $V_1{_{|\Omega_0}}\in L^n(\Omega_0)$. Let $V_2\in L^{n/2}(\Omega)$ such that $V_2=V_1$ in $\Omega_0$ and
\[
\gamma_1u_{V_1}(\varphi)(\cdot,t_\ast)_{|\Gamma_1}=\gamma_1u_{V_2}(\varphi)(\cdot,t_\ast)_{|\Gamma_1},\quad  \varphi\in \mathcal{B}_0.
\]
Then $V_1=V_2$.
\end{theorem}

Theorems \ref{thm1} and \ref{thm2} can be extended to $n=2,3,4$ by making the modifications explained in the comments after \cite[Theorem 1.2]{Ch24}.

There is a large literature devoted to parabolic inverse problems. We refer the reader to \cite{Ch,Is}, where he can find typical examples of parabolic inverse problems.

\section{Solving the non-homogenous IBVP}

We need to construct a family of one-parameter operators, necessary to establish the existence and uniqueness of the solutions of the IBVP \eqref{p1}. To do this, we recall some results concerning the spectral decomposition of the Schr\"odinger operator $-\Delta+V$, $V\in L^{n/2}(\Omega)$, under Dirichlet boundary condition. We define on $H^1(\Omega)\times H^1(\Omega )$ the bilinear form $\mathfrak{a}_V$ associated with $V\in L^{n/2}(\Omega)$ by
\[
\mathfrak{a}_V(u,v)=\int_\Omega \left(\nabla u\cdot \nabla v+Vuv\right)dx.
\]
Let $A_V: H_0^1(\Omega )\rightarrow H^{-1}(\Omega )$ be the bounded operator defined by 
\[
\langle A_Vu,v\rangle =\mathfrak{a}_V(u,v),\quad u,v\in H_0^1(\Omega ).
\]
We proved in \cite{Ch24} that the spectrum of $A_V$, denoted by $\sigma(A_V)$, consists of a sequence $(\lambda_V^k)$ satisfying
\[
-\infty< \lambda_V^1\le \lambda_V^2\le \ldots \le \lambda_V^k\le \ldots
\]
and
\[
\lambda_V^k\rightarrow \infty \quad  \mbox{as}\; k\rightarrow \infty .
\]

Moreover, $L^2(\Omega )$ admits an orthonormal basis $(\phi_V^k)$ of eigenfunctions, each $\phi_V^k$ being associated with $\lambda_V^k$, which means that $\phi_V^k\in H_0^1(\Omega)$ and
\[
\int_\Omega \left(\nabla \phi_V^k\cdot \nabla v+V\phi_V^k v\right)dx=\lambda_V^k\int_\Omega \phi_V^k vdx, \quad v\in H_0^1(\Omega).
\]
In particular, $(-\Delta +V)\phi_V^k=\lambda_V^k\phi_V^k$ in the distributional sense, and since $(\lambda_V^k-V)\phi_V^k\in L^q(\Omega)$, we obtain from the usual regularity $W^{2,q}$ that $\phi_V^k\in W^{2,q}(\Omega)$.

For simplicity, we hereafter use the notation
\[
\psi_V^k=\gamma_1 \phi_V^k,\quad k\ge 1,\; V\in L^{n/2}(\Omega).
\]

We fix $V_0\in  L^{n/2}(\Omega)$ nonnegative and non identically equal to zero and we define
\[
\mathscr{V}=\left\{V\in  L^{n/2}(\Omega);\; |V|\le V_0\right\}.
\]
The following bilateral inequality was established in \cite[Proposition 2.1]{Ch24} 
\begin{equation}\label{pa1}
\mathfrak{c}_0^{-1}k^{2/n}-\mathfrak{c}_1\le  \lambda_V^k \le \mathfrak{c}_0k^{2/n}+\mathfrak{c}_1,\quad k\ge 1, V\in \mathscr{V}.
\end{equation}
Here and henceforth, $\mathfrak{c}_0=\mathfrak{c}_0(\Omega)\ge 1$ and $\mathfrak{c}_1=\mathfrak{c}_1(n,\Omega,V_0)>0$ are constants.

We use the following notations below
\[
\mathbb{C}_+=\{z\in \mathbb{C};\; \Re z\ge 0\},\quad  \mathbb{C}_+^0=\{z\in \mathbb{C};\; \Re z>0\}.
\]
Let $V\in \mathscr{V}$ and  define on $L^2(\Omega)$ the family of linear operator  $(\mathbb{T}_V(z))_{z\in \mathbb{C}_+}$  as follows
\begin{equation}\label{pa2}
\mathbb{T}_V(z) f=\sum_{k\ge 1}e^{-\lambda_V^kz}(f|\phi_V^k)\phi_V^k,\quad f\in L^2(\Omega),\; z\in \mathbb{C}_+.
\end{equation}
Here and henceforth, $(\cdot|\cdot)$ is the usual inner product of $L^2(\Omega)$. The norm of $L^2(\Omega)$ will denoted by $\|\cdot\|_2$.

We gather in the following lemma the properties of $(\mathbb{T}_V(z))_{z\in \mathbb{C}_+}$ that we need.

\begin{lemma}\label{lem1}
Let $V\in \mathscr{V}$. Then
\\
{\rm (1)} $\mathbb{T}_V(0)=I_{L^2(\Omega)}$.
\\
{\rm (2)} For all $z\in \mathbb{C}_+$, $\mathbb{T}_V(z)\in \mathscr{B}(L^2(\Omega))$ and
\begin{equation}\label{pa2.1}
\|\mathbb{T}_V(z) f\|\le e^{\mathfrak{c}_1\Re z}\|f\|_2,\quad f\in L^2(\Omega),\; z\in \mathbb{C}_+.
\end{equation}
{\rm (3)} For all $z_1,z_2\in \mathbb{C}_+$, we have
\[
\mathbb{T}_V(z_1)\mathbb{T}_V(z_2)=\mathbb{T}_V(z_1+z_2).
\]
{\rm (4)} For all $f\in L^2(\Omega)$, $z\in \mathbb{C}_+\mapsto \mathbb{T}_Vf$ is continuous.
\\
{\rm (5)} For all $f\in L^2(\Omega)$, $z\in \mathbb{C}_+^0\mapsto \mathbb{T}_V(z)f\in L^2(\Omega)$ is holomorphic and 
\begin{equation}\label{pa3}
\frac{d^m}{dz^m}\mathbb{T}_V(z)f=\sum_{k\ge 1}(-\lambda_V^k)^me^{-\lambda_V^kz}(f|\phi_V^k)\phi_V^k,\quad m\ge 0.
\end{equation}
{\rm (6)} For all $f\in L^2(\Omega)$, we have
\begin{equation}\label{pa3.2}
\left\| \frac{d}{dz}[\mathbb{T}_V(z)f]\right\|_2\le (\mathbf{c}_0|\Re z|^{-1}+1)e^{\mathfrak{c}_1\Re z}\|f\|_2,\quad z\in \mathbb{C}_+^0,
\end{equation}
where $\mathbf{c}_0=\sup_{\rho >0}\rho e^{-\rho}$.
\\
{\rm (7)} For all $z\in \mathbb{C}_+^0$ and $f\in L^2(\Omega)$, the series $\sum_{k\ge 1}e^{-\lambda_V^kz}(f|\phi_V^k)\phi_V^k$ is norm convergent in $W^{2,q}(\Omega)$ and then $\gamma_0\mathbb{T}_V(z)=0$. Furthermore, the following inequality holds
 \begin{equation}\label{pa4.1}
\| \mathbb{T}_V(z)f\|_{W^{2,q}(\Omega)}\le \mathbf{c}(|\Re z|^{-1}+1)e^{\mathfrak{c}_1\Re z}\|f\|_2,
 \end{equation}
 where $\mathbf{c}=\mathbf{c}(n,\Omega,V_0)>0$ is a constant.
 \\
 {\rm (8)} For all $z\in \mathbb{C}_+^0$ and $f\in L^2(\Omega)$, $V\mathbb{T}_V(z)f\in L^q(\Omega)$ and
 \begin{equation}\label{pa5}
\left(\frac{d}{dz}-\Delta +V\right)\mathbb{T}_V(z)f=0\quad \mathrm{in}\; \mathbb{C}_+^0.
\end{equation}
\end{lemma}

\begin{proof}
(1) is obvious and (2) follows  from \eqref{pa1}.
\\
(3) Let $z_1,z_2\in \mathbb{C}_+$ and $f\in L^2(\Omega)$. By taking in \eqref{pa2} $z=z_1$ and $f=\mathbb{T}_V(z_2)f$,  we get
\begin{align*}
\mathbb{T}_V(z_1)\mathbb{T}_V(z_2)f&= \sum_{k\ge1}e^{-\lambda_V^kz_1}\sum_{\ell\ge 1}e^{-\lambda_V^\ell z_2}(f|\phi_V^\ell)(\phi_V^\ell|\phi_V^k)\phi_V^k
\\
&=\sum_{k\ge 1}e^{-\lambda_V^k(z_1+z_2)}(f|\phi_V^k)\phi_V^k.
\end{align*}
Thus,
\[
\mathbb{T}_V(z_1)\mathbb{T}_V(z_2)=\mathbb{T}_V(z_1+z_2).
\]
(4) Let $\epsilon >0$, $f\in L^2(\Omega)$ and $\ell \ge 1$ sufficiently large in such a way that $\lambda_V^k\ge 0$, $k\ge \ell$ and 
\[
\sum_{k\ge \ell}|(f|\phi_V^k)|^2\le  \epsilon.
\]
With this choice, we obtain
\[
\|\mathbb{T}_V(z_1)f-\mathbb{T}_V(z_2)f\|_2^2=\sum_{k\le\ell}|e^{-\lambda_V^kz_1}-e^{-\lambda_V^kz_2}|^2|(f|\phi_V^k)|^2+\epsilon,
\]
from which we obtain that $z\in \mathbb{C}_+\mapsto \mathbb{T}_V(z_1)f\in L^2(\Omega)$ is continuous.
\\
(5) Let $z_0\in \mathbb{C}_+^0$. Then we have for $|z-z_0|\le \Re z_0/2$ and $f\in L^2(\Omega)$
\[
|e^{-\lambda_V^k z}|\le e^{-\lambda_V^k \Re z_0/2}\quad \mathrm{if}\;  \lambda_V^k \ge 0,
\]
and therefore the series $\sum_{k\ge 1}e^{-\lambda_V^kz}(f|\phi_V^k)\phi_V^k$ converges in $L^2(\Omega)$, uniformly in $B(z_0,\Re z_0/2)$, which implies that $z\in \mathbb{C}_+^0\mapsto \mathbb{T}_V(z)f\in L^2(\Omega)$ is holomorphic and 
\begin{equation}
\frac{d^m}{dz^m}\mathbb{T}_V(z)f=\sum_{k\ge 1}(-\lambda_V^k)^me^{-\lambda_V^kz}(f|\phi_V^k)\phi_V^k,\quad m\ge 0.
\end{equation}
(6) Since
\[
e^{-\mathfrak{c}_1z}\mathbb{T}_V(z)f=\sum_{k\ge 1}e^{-(\lambda_V^k+\mathfrak{c_1})z}(f|\phi_V^k)\phi_V^k,
\]
we have
\[
\frac{d}{dz}[e^{-\mathfrak{c}_1z}\mathbb{T}_V(z)f]=-\sum_{k\ge 1}(\lambda_V^k+\mathfrak{c_1})e^{-(\lambda_V^k+\mathfrak{c_1})z}(f|\phi_V^k)\phi_V^k.
\]
Hence
\begin{equation}\label{pa3.1}
\left\| \frac{d}{dz}[e^{-\mathfrak{c}_1z}\mathbb{T}_V(z)f]\right\|_2\le \mathbf{c}_0|\Re z|^{-1}\|f\|_2,
\end{equation}
where $\mathbf{c}_0=\sup_{\rho > 0}\rho e^{-\rho}$. By using
\[
\frac{d}{dz}\mathbb{T}_V(z)f= \mathfrak{c}_1\mathbb{T}_V(z)f +e^{\mathfrak{c}_1z}\frac{d}{dz}[e^{-\mathfrak{c}_1z}\mathbb{T}_V(z)f],
\]
we obtain from \eqref{pa2.1} and \eqref{pa3.1}
\[
\left\| \frac{d}{dz}[\mathbb{T}_V(z)f\right\|_2\le (\mathbf{c}_0|\Re z|^{-1}+1)e^{\mathfrak{c}_1\Re z}\|f\|_2.
\]
(7) We already know from \cite{Ch24} that, for all $k\ge 1$, $\phi_V^k\in W^{2,q}(\Omega)$ and $\|\phi_V^k\|_{W^{2,q}(\Omega)}\le \mathbf{c}(1+|\lambda_V^k|)$, where $\mathbf{c}=\mathbf{c}(n,\Omega,V_0)>0$ is a constant. As above, \eqref{pa1} allows us to deduce that, for all $z\in \mathbb{C}_+^0$, the series $\sum_{k\ge 1}e^{-\lambda_V^kz}(f|\phi_V^k)\phi_V^k$  is norm convergent in $W^{2,q}(\Omega)$. We then proceed as for $\frac{d}{dz}[\mathbb{T}_V(z)]f$ to prove \eqref{pa4.1}.
\\
(8) By using that $\|V \phi_V^k\|_{L^q(\Omega)}\le \|V_0\|_{L^{n/2}(\Omega)}$, we deduce that the series in the right hand side of \eqref{pa2} is norm convergent in $L^q(\Omega)$ and
\[
\sum_{k\ge 1}e^{-\lambda_V^kz}(f|\phi_V^k)V \phi_V^k=V\sum_{k\ge 1}e^{-\lambda_V^kz}(f|\phi_V^k) \phi_V^k=V \mathbb{T}_V(z)f,\quad z\in \mathbb{C}_+^0.
\]
As
\[
\Delta \mathbb{T}_V(z)f=\sum_{k\ge 1}e^{-\lambda_V^kz}(f|\phi_V^k)\Delta \phi_V^k, \quad z\in \mathbb{C}_+^0,
\]
we obtain from \eqref{pa3}
\[
\left(\frac{d}{dz}-\Delta +V\right)\mathbb{T}_V(z)f=\sum_{k\ge 1}(f|\phi_V^k)(-\lambda_V^k-\Delta +V)\phi_V^k=0\quad \mathrm{in}\; \mathbb{C}_+^0.
\]
The proof is then complete.
\end{proof}

The following consequence of Lemma \ref{lem1} will be useful later.

\begin{corollary}\label{cor1}
Let $f\in L^2(\Omega)$. Then
\[
u(\cdot,t):=\mathbb{T}_V(t)f\in C([0,\infty),L^2(\Omega))\cap C^1(]0,\infty[,L^q(\Omega))\cap C(]0,\infty[,W^{2,q}(\Omega))
\]
 is the unique solution of the following IBVP
\begin{equation}\label{pa6}
\left\{
\begin{array}{ll}
(\partial_t -\Delta +V)u=0\quad &\mathrm{in}\; Q:=\Omega\times ]0,\infty[,
\\
\gamma_0u=0 &\mathrm{on}\; \Sigma :=\Gamma \times ]0,\infty[,
\\
u(\cdot,0)=f.
\end{array}
\right.
\end{equation}
\end{corollary}

\begin{proof}
We need only to prove the uniqueness. Let then 
\[
u\in C([0,\infty),L^2(\Omega))\cap C^1(]0,\infty[,L^q(\Omega))\cap C(]0,\infty[,W^{2,q}(\Omega))
\]
satisfying \eqref{pa6} and $u_k(t)=(u(\cdot,t)|\phi_V^k)$, $t\ge 0$ and $k\ge 1$. Applying Green's formula, we obtain that $u_k$, $k\ge 1$, must be the solution of following differential equation
\[
u_k'(t)+\lambda_V^ku_k(t)=0,\; t>0,\quad u_k(0)=(f|\phi_V^k).
\]
Therefore $u_k(t)=e^{-\lambda_V^kt}(f|\phi_V^k)$, which gives $u(\cdot,t)=\mathbb{T}_V(t)f$.
\end{proof}

We set 
\[
\mathcal{F}=\{F\in W^{1,1}((0,\mathfrak{t}),L^2(\Omega))\cap C^\alpha([0,t],L^2(\Omega));\; \mathrm{supp}\, F\subset [t_0,\mathfrak{t}]\times \overline{\Omega}\}.
\]

\begin{proposition}\label{pro1}
Let $V\in \mathscr{V}$. For all $F\in \mathcal{F}$ 
\[
u_V(F)(\cdot,t)=\int_0^t\mathbb{T}_V(t-s)F(\cdot,s)ds, \quad 0\le t\le \mathfrak{t},
\]
belongs to $C^1([0,T],L^2(\Omega))\cap C([0,T],W^{2,q}(\Omega))$ and  is the unique solution of the IBVP
\begin{equation}\label{pa8}
\left\{
\begin{array}{ll}
(\partial_t -\Delta +V)u=F\quad &\mathrm{in}\; Q_{\mathfrak{t}},
\\
\gamma_0u=0 &\mathrm{on}\; \Sigma_{\mathfrak{t}},
\\
u(\cdot,0)=0.
\end{array}
\right.
\end{equation}
\end{proposition}

\begin{proof} 
In this proof, $\mathbf{c}=\mathbf{c}(n,\Omega,V_0)$ is a generic constant and, for simplicity, we set $u:=u_V(F)$. We have
\begin{align*}
\partial_t\mathbb{T}_V(t-s)F(\cdot,s)&=-\partial_s\mathbb{T}_V(t-s)F(\cdot,s)
\\
&=-\partial_s[\mathbb{T}_V(t-s)F(\cdot,s)ds]+\mathbb{T}_V(t-s)\partial_sF(\cdot,s).
\end{align*}
Hence, $s\in (0,t)\mapsto \partial_t\mathbb{T}_V(t-s)F(\cdot,s)$ belongs to $L^1((0,t),L^2(\Omega))$ and
\[
 \int_0^t\partial_t\mathbb{T}_V(t-s)F(\cdot,s)ds=-F(\cdot,t) -\int_0^t\mathbb{T}_V(t-s)\partial_sF(\cdot,s)ds.
\]
In other words,  $u\in C^1([0,\mathfrak{t}],L^2(\Omega))$ and
\[
\partial_tu(\cdot,t)=-\int_0^t\mathbb{T}_V(t-s)\partial_sF(\cdot,s)ds.
\]
Whence
\[
\|\partial_tu(\cdot ,t)\|_{L^2(\Omega)}\le \mathbf{c}e^{\mathfrak{c}_1\mathfrak{t}}\|\partial_tF\|_{L^1((0,\mathfrak{t}),L^2(\Omega))},
\]
and then
\[
\|u(\cdot ,t)\|_{C^1([0,\mathfrak{t}],L^2(\Omega))}\le \mathbf{c}e^{\mathfrak{c}_1\mathfrak{t}}\|F\|_{W^{1,1}((0,\mathfrak{t}),L^2(\Omega))},
\]

On the other hand, as 
\[
\|V\mathbb{T}_V(t-s)F(\cdot,s)\|_{L^q(\Omega)}\le \|V\|_{L^{n/2}(\Omega)} e^{\mathfrak{c}_1(t-s)}\|F(\cdot,s)\|_2,
\]
we deduce that $s\in (0,t)\mapsto V\mathbb{T}_V(t-s)F(\cdot,s)\in L^1((0,t),L^q(\Omega))$. 

We have $u=0$ in $[0,t_0]$ and 
\[
u(\cdot , t)=\int_{t_0}^t\mathbb{T}_V(t-s)F(\cdot,s)ds,\quad t>t_0.
\]

We write $u$ is the form $u=u_1+u_2$, where
\[
u_1(\cdot , t)=\int_{t_0}^t\mathbb{T}_V(t-s)(F(\cdot,s)-F(\cdot,t))ds ,\quad u_2(\cdot,t)=\int_{t_0}^t\mathbb{T}_V(s)F(\cdot,t)ds.
\]
It follows from \eqref{pa4.1}
\begin{align*}
&\|\mathbb{T}_V(t-s)(F(\cdot,s)-F(\cdot,t))\|_{W^{2,q}(\Omega)}
\\
&\hskip 3cm \le \mathbf{c}((t-s)^{-1}+1)(t-s)^\alpha e^{\mathfrak{c}_1t}\|F\|_{C^\alpha([0,T],L^2(\Omega))}.
\end{align*}
By using again \eqref{pa4.1}, we get
\[
 \int_{t_0}^t \|\mathbb{T}_V(s) F(\cdot,t)\|_{W^{2,q}(\Omega)}\le \mathbf{c}\|F(\cdot,t)\|_{L^2(\Omega)}\int_{t_0}^t(s^{-1}+1)ds.
\]
Thus, $u\in C([0,\mathfrak{t}], W^{2,q}(\Omega))$ and, as $\Delta :W^{2,q}(\Omega)\rightarrow L^q(\Omega)$ is bounded, we get
\[
\Delta u(\cdot,t)=\int_0^t \Delta \mathbb{T}_V(s) F(\cdot,s) ds.
\]
But, we already know that
\[
(\partial_t-\Delta +V)\mathbb{T}_V(t-s)F(\cdot,s)=0,\quad 0\le s<t.
\]
We end up concluding that  $u\in C^1([0,T],L^2(\Omega))\cap C([0,T],W^{2,q}(\Omega))$  is the unique solution of the IBVP of \eqref{pa8}.
\end{proof}

We use Proposition \ref{pro1} to prove the following result.

\begin{proposition}\label{pro2}
Let $V\in \mathscr{V}$ and $\varphi \in \mathcal{B}$. Then the IBVP
\begin{equation}\label{pa9}
\left\{
\begin{array}{ll}
(\partial_t -\Delta +V)u=0\quad &\mathrm{in}\; Q_{\mathfrak{t}},
\\
\gamma_0u=\varphi &\mathrm{on}\; \Sigma_{\mathfrak{t}},
\\
u(\cdot,0)=0
\end{array}
\right.
\end{equation}
admits a unique solution $u_V(\varphi)\in C^1([0,\mathfrak{t}],L^2(\Omega))\cap C([0,\mathfrak{t}],W^{2,q}(\Omega))$ given by
\begin{equation}\label{pa10}
u_V(\varphi)(\cdot, t)=-\sum_{k\ge 1}\int_0^t e^{-\lambda_V^k(t-s)}\langle \varphi(\cdot,s)|\psi_V^k\rangle \phi_V^kds,
\end{equation}
where
\[
\langle \varphi(\cdot,t)|\psi_V^k\rangle=\int_\Gamma \varphi(\cdot,t)\psi_V^kd\sigma.
\]
\end{proposition}

\begin{proof}
As for \eqref{pa6}, we see that \eqref{pa9} has at most one solution in $C^1([0,\mathfrak{t}],L^2(\Omega))\cap C([0,\mathfrak{t}],W^{2,q}(\Omega))$, and  if $G\in \dot{\varphi}$ is chosen arbitrarily and $F= -(\partial_t-\Delta+V)G$, then we check that $F\in \mathcal{F}$ and $u=G+u_V(F)\in C^1([0,\mathfrak{t}],L^2(\Omega))\cap C([0,\mathfrak{t}],W^{2,q}(\Omega))$ is a solution of \eqref{pa9}.

Next, let $u_k(t)=(u(\cdot ,t)|\phi_V^k)$, $t\ge 0$ and $k\ge 1$. Then, $u_k(0)=0$ and
\[
u_k'(t)+((-\Delta +V)u(\cdot,t)|\phi_V^k)=0.
\]
We obtain from Green's formula
\[
u_k'(t)+\lambda_ku_k(t)=-\langle \varphi(\cdot,t)|\psi_V^k\rangle.
\]
Thus,
\[
u_k(t)=-\int_0^t e^{-\lambda_V^k(t-s)}\langle \varphi(\cdot,s)|\psi_V^k\rangle ds
\]
and therefore
\[
u(\cdot, t)=\sum_{k\ge 1}\int_0^t e^{-\lambda_V^k(t-s)}\langle \varphi(\cdot,s)|\psi_V^k\rangle \phi_V^kds.
\]
That is we proved \eqref{pa10}.
\end{proof}

\section{Proof of Theorems \ref{thm1} and \ref{thm2}}

Let $p=2n/(n+2)$. We need following uniqueness of continuation result.

\begin{lemma}\label{lem2}
Let $\tilde{\Gamma}$ be a nonempty open subset of $\Gamma$ and $W\in L^{n/2}(\Omega)$. If $v\in W^{2,p}(\Omega)$ satisfies $|\Delta v|\le |W||v|$ in $\Omega$ and $\gamma_0v_{|\tilde{\Gamma}}=\gamma_1v_{|\tilde{\Gamma}}=0$, then $v=0$. 
\end{lemma}

\begin{proof}
Let $\tilde{x}\in \tilde{\Gamma}$ and $B$ be a ball centered at $\tilde{x}$ chosen in such a way that $B\cap \Gamma \subset \tilde{\Gamma}$. We extend $W$ and $v$ in $B\setminus \Omega$ by $0$. These extensions are still denoted by $W$ and $v$, respectively. We verify that $v\in W^{2,p}(\Omega \cup B)$ and  $|\Delta v|\le |W||v|$ in $\Omega\cup B$. Then $v=0$ by \cite[Theorem 6.3]{JK}.
\end{proof}

\begin{corollary}\label{cor2}
Let $\tilde{\Gamma}$ be a nonempty open subset of $\Gamma$, $V\in L^{n/2}(\Omega)$, $\lambda\in \sigma(A_V)$ and  $\phi\in H_0^1(\Omega)$ be a linear combination of linearly independent eigenfunctions corresponding to the eigenvalue $\lambda$. If $\gamma_1\phi_{|\tilde{\Gamma}}=0$, then $\phi=0$.
\end{corollary}

\begin{proof}
As $H_0^1(\Omega)$ is continuously embedded in $L^{p'}(\Omega)$, $p'=2n/(n-2)$, we have $-\Delta \phi= (\lambda-V)\phi\in L^p(\Omega)$. According to regularity $W^{2,p}$, $\phi \in W^{2,p}(\Omega)$. We complete the proof by applying Lemma \ref{lem2}.
\end{proof}

Another ingredient we need is the following algebraic lemma borrowed from \cite{CK}. 

\begin{lemma}\label{lem3}
Let $S_1$, $S_2$ be two sets so that $S_1\cap S_2$ contains infinitely many points, and set $S=S_1\cup S_2$. For $j=1,2$, let $(f_1^j,\dots ,f_{m_j}^j):S\rightarrow \mathbb{R}^{m_j}$ such that $f_1^j,\dots ,f_{m_j}^j$ are linearly independent in $S$ and $f_k^j$ is non identically equal to zero in $S_1\cap S_2$, $1\le k\le m_j$. If
\[
\sum_{k=1}^{m_1}f_k^1(\xi)f_k^1(\eta)=\sum_{k=1}^{m_2}f_k^2(\xi)f_k^2(\eta),\quad (\xi,\eta)\in S_1\times S_2,
\]
then $m_1=m_2$ and there exists a $m_1\times m_1$ orthogonal  matrix $P$ so that
\[
(f_1^1,\dots ,f_{m_1}^1)^t=P(f_2^2,\dots ,f_{m_1}^2)^t.
\]
\end{lemma}

\begin{proof}[Proof of Theorem \ref{thm1}]
Let $V_1,V_2\in L^{n/2}(\Omega)$ satisfying
\begin{equation}\label{par1}
\gamma_1u_{V_1}(\varphi)(\cdot,t_\ast)_{|\Gamma_1}=\gamma_1u_{V_2}(\varphi)(\cdot,t_\ast)_{|\Gamma_1}, \quad \varphi\in \mathcal{B}_0.
\end{equation}
Let $j=1,2$ and $\varphi\in \mathcal{B}_0$. From Proposition \ref{pro2}, $u_{V_j}(\varphi)\in C([0,\mathfrak{t}],W^{2,q}(\Omega))$ and 
\begin{equation}\label{pa11}
\gamma_1u_{V_j}(\cdot, t_\ast)(\varphi)_{|\Gamma_1}=-\sum_{k\ge 1}\int_0^{t_\ast} e^{-\lambda_{V_j}^k(t_\ast-s)}\langle \varphi(\cdot,s)|\psi_{V_j}^k\rangle \psi_{V_j}^k{_{|\Gamma_1}} ds.
\end{equation}
In light of \eqref{par1}, we have
\begin{equation}\label{pa12}
\int_{\Gamma_0\times (0,t_\ast-\epsilon)} \varphi (\xi,s)\Phi(\xi,\eta,s) d\sigma(\xi)ds=0, \quad \mathrm{a.e}\; \eta \in \Gamma_1,
\end{equation}
where, for all $(\xi,\eta,s)\in \Gamma_0\times \Gamma_1\times (0,t_\ast-\epsilon)$,
\[
\Phi(\xi,\eta,s)=\sum_{k\ge 1}e^{-\lambda_{V_1}^k(t_\ast-s)}\psi_{V_1}^k(\xi){_{|\Gamma_1}}\psi_{V_1}^k(\eta){_{|\Gamma_1}}-\sum_{k\ge 1}e^{-\lambda_{V_2}^k(t_\ast-s)}\psi_{V_2}^k(\xi){_{|\Gamma_1}}\psi_{V_2}^k(\eta){_{|\Gamma_1}}.
\]
Since \eqref{pa12} holds for all $\varphi \in \mathcal{B}_0$,  we deduce that $\Phi=0$. That is we have for all $(\xi,\eta,s)\in \Gamma_0\times \Gamma_1\times (0,t_\ast-\epsilon)$
\[
\sum_{k\ge 1}e^{-\lambda_{V_1}^k(t_\ast-s)}\psi_{V_1}^k(\xi){_{|\Gamma_1}}\psi_{V_1}^k(\eta){_{|\Gamma_1}}=\sum_{k\ge 1}e^{-\lambda_{V_2}^k(t_\ast-s)}\psi_{V_2}^k(\xi){_{|\Gamma_1}}\psi_{V_2}^k(\eta){_{|\Gamma_1}}.
\]
We rewrite this identity as equality between two Dirichlet series. Instead of $(\lambda_{V_j}^k)$, we consider the sequence of distinct eigenvalues that we still denote by $(\lambda_{V_j}^k)$, $j=1,2$. We gather together eigenfunctions associated with the same eigenvalue that we denote by $\phi_{V_j}^{k,1},\ldots \phi_{V_j}^{k,m_k^j}$, where $m_k^j$ is the multiplicity of $\lambda_{V_j}^k$, $k\ge 1$, $j=1,2$. Then the preceding identity takes the form
\[
\sum_{k\ge 1}e^{-\lambda_{V_1}^k(t_\ast-s)}\sum_{\ell=1}^{m_k^1}\psi_{V_1}^{k,\ell}(\xi){_{|\Gamma_1}}\psi_{V_1}^{k,\ell}(\eta){_{|\Gamma_1}}=\sum_{k\ge 1}e^{-\lambda_{V_2}^k(t_\ast-s)}\sum_{\ell=1}^{m_k^2}\psi_{V_2}^{k,\ell}(\xi){_{|\Gamma_1}}\psi_{V_2}^{k,\ell}(\eta){_{|\Gamma_1}}.
\]
According to the uniqueness of Dirichlet series, we obtain, for each $k\ge 1$, $\lambda_{V_1}^k=\lambda_{V_2}^k$ and
\[
\sum_{\ell=1}^{m_k^1}\psi_{V_1}^{k,\ell}{_{|\Gamma_1}}(\xi)\psi_{V_1}^{k,\ell}(\eta){_{|\Gamma_1}}=\sum_{\ell=1}^{m_k^2}\psi_{V_2}^{k,\ell}{_{|\Gamma_1}}(\xi)\psi_{V_2}^{k,\ell}(\eta){_{|\Gamma_1}}.
\]
In light of Corollary \ref{cor2}, we can apply Lemma \ref{lem3} to $(\psi_{V_1}^{k,1}{_{|\Gamma_1}},\ldots ,\psi_{V_1}^{k,m_k^1}{_{|\Gamma_1}})$ and $(\psi_{V_2}^{k,1}{_{|\Gamma_1}},\ldots ,\psi_{V_2}^{k,m_k^2}{_{|\Gamma_1}})$. By replacing $(\psi_{V_2}^k)$ with another orthonormal basis of $L^2(\Omega)$, we end up obtaining
\begin{equation}\label{S1}
\lambda_{V_1}^k=\lambda_{V_2}^k,\quad \psi_{V_1}^k=\psi_{V_2}^k,\quad k\ge 1.
\end{equation}
Then $V_1=V_2$ by \cite[Theorem 1.1]{Po}.
\end{proof}

\begin{proof}[Proof of Theorem \ref{thm2}]
Let $\tilde{\Gamma}=\Gamma_0\cup \Gamma_1$. By proceeding as in the previous proof, we get instead of \eqref{S1} 
\begin{equation}\label{S2}
\lambda_{V_1}^k=\lambda_{V_2}^k,\quad \psi_{V_1}^k{_{|\tilde{\Gamma}}}=\psi_{V_2}^k{_{|\tilde{\Gamma}}},\quad k\ge 1.
\end{equation}
By applying \cite[Corollary 1.1]{Ch24}, we obtain $V_1=V_2$.
\end{proof}

\begin{remark}\label{rem1}
{\rm
Let $\Omega$ be a $C^\infty$ smooth domain. Let $\mathbf{g}=(g_{k\ell})$ be a Riemannian metric $C^\infty$ in $\overline{\Omega}$. We assume that $\mathbf{g}$ is simple in $\Omega$, which means that $\Omega$ is strictly convex with respect to the metric $\mathbf{g}$ and for all $x\in \overline{\Omega}$ the exponential map $\exp_x:\exp_x^{-1}(\overline{\Omega})\rightarrow \overline \Omega$ is a diffeomorphism. Theorems \ref{thm1} and \ref{thm2} then remain valid when the Laplace operator in \eqref{p1} is replaced by the Laplace-Beltrami operator associated with the metric $\mathbf{g}$. We proceed exactly as in the proof of Theorems \ref{thm1} and \ref{thm2} to obtain \eqref{S1} and \eqref{S2}. The proof is completed by using \cite[Theorem 4.1]{Ch24} for \eqref{S1}. For \eqref{S2}, we additionally use in $\Omega_0$ the unique continuation from the Cauchy data.
}
\end{remark}

\section*{Second part: Determination of the initial condition from a single interior measurement}

\section{Introduction}

In this part, $\Omega$ denotes a bounded Lipschitz domain of $\mathbb{R}^n$, $n\ge 3$. Assume that there exists $\delta_0>0$ so that
\[
\{z\in \Omega \times (0,1);\; \mathrm{dist}(z,\partial(\Omega \times (0,1)))>\delta\}
\]
is connected whenever $0<\delta <\delta_0$.

Let $p:=2n/(n+2)$, $p':=2n/(n-2)$ its conjugate ($1/p+1/p'=1$) and 
\[
\sigma:=\sup \{ \|w\|_{L^{p'}(\Omega)};\; w\in H_0^1(\Omega),\;  \|\nabla w\|_{L^2(\Omega)}=1\},
\]
and fix $\vartheta>0$ in such a way that $2\vartheta\sigma^2<1$. Set
\[
\mathcal{V}:=\{ V\in L^{n/2}(\Omega);\; \|V\|_{L^{n/2}(\Omega)}\le \vartheta\}.
\]

Henceforth, $V\in \mathcal{V}$, $0<s<1/2$, $\beta >0$ and $\omega\Subset \Omega$ are fixed. The quantities $(\lambda_V^j)$, $(\phi_V^j)$ and $\mathbb{T}_V(t)$ are those defined in the first part.

Set
\begin{align*}
&N_\beta(f):=\left(\sum_{j\ge 1}(1+|\lambda_V^j|)^{2\beta}|(f|\phi_V^j)|^2\right)^{1/2},
\\
&N(f)=\sum_{j\ge 1}(1+|\lambda_V^j|)^{1/2}|(f|\phi_V^j)|+\left(\sum_{j\ge 1}(1+|\lambda_V^j|)(f|\phi_V^j)|^2\right)^{1/2},
\end{align*}
and define
\begin{align*}
&\mathcal{H}=\{f\in L^2(\Omega);\; N(f)<\infty\}, 
\\
&\mathcal{H}_0=\{f\in \mathcal{H};\; N_\beta(f)<\infty\}
\end{align*}

Note that $\mathcal{H}_0=\mathcal{H}$ whenever $0<\beta \le 1/2$.

Let  $\mathbf{D}=\mathrm{diam}(\Omega)$, $\rho_0=\rho_0(\omega,\mathbf{D})$ and $\hat{c}=\hat{c}(n,\mathbf{D},s)$ be the constants in Theorem \ref{gquc} of Appendix \ref{appA} when $\Omega$ and $\omega$ are replaced by $\Omega\times (0,1)$ and $\omega\times (1/4,3/4)$, respectively. Set $H_\rho=e^{e^{\hat{c}\rho^{-n}}}$, $0<\rho <\rho_0$, and
\[
L_\rho (u)=H_\rho\|u\|_{L^2(\omega)}+\rho^s \|u\|_{H^1(\Omega)},\quad u\in H^1(\Omega),\; 0<\rho <\rho_0.
\]
Also, let $\varphi (\rho)=\rho^{-s}e^{e^{\hat{c}\rho^{-n}}}$, $0<\rho<\hat{\rho}$, where $0<\hat{\rho}\le \rho_0$ is chosen so that $\varphi(\hat{\rho})>e$. Define
\[
\Phi (\ell)=\chi_{(\varphi(\hat{\rho}),\infty)}(\ln \ln \ell )^{-s}+\chi_{(0,\varphi(\hat{\rho})]} (H_{\hat{\rho}}+\hat{\rho}^s\varphi(\hat{\rho}))\ell ^{-1}.
\]
Here $\chi_J$ denotes the characteristic function of the intervalle $J$.

In the rest of this part, $\mathbf{c}=\mathbf{c}(n,\Omega,\omega,V,s)$ will denote a generic constant.

Our objective in this second part is to prove the following result, where $\mathbf{c}_\mathfrak{t}$ is a constant of the form
\[
\mathbf{c}_\mathfrak{t}=\mathbf{c}e^{10\mathfrak{t}^{-1}+\mathfrak{c}_1\mathfrak{t}}.
\]
Here and henceforth, $\mathfrak{c}_1$ is as in \eqref{pa1} of the first part with $V_0=|V|$.

\begin{theorem}\label{theoremin1}
For all $f\in \mathcal{H}_0\setminus\{0\}$ and $\lambda>0$ we have 
\begin{equation}\label{in4}
\|f\|_{L^2(\Omega)} \le \mathbf{c}_\mathfrak{t}e^{\lambda \mathfrak{t}}\|\mathbf{I}(f)\|_{L^1((0,\mathfrak{t}))}+\lambda^{-\beta}[f]_\beta ,
\end{equation}
where
\[
\mathbf{I}(f)(t):=N(f)\Phi\left(N(f)/\|T_V(t)f\|_{L^2(\omega)}\right),\quad t\in (0,\mathfrak{t}).
\]
\end{theorem}

Theorem \ref{theoremin1} clearly quantifies the unique determination of $f$ from the interior data $\mathbb{T}_V(t)f{_{|\omega\times (0,\mathfrak{t})}}$.

Let
\[
\Psi(r)=r^{-1}\chi_{]0,1]}+(\log r)^{-\beta}\chi_{]1,\infty[}.
\]
Theorem \ref{theoremin1} then yields the following  stability inequality.

\begin{corollary}\label{corollaryin1}
There exists a constant $\mathfrak{c}=\mathfrak{c}(n,\Omega,\omega,V,\mathfrak{t},\beta)>0$ such that for all $f\in \mathcal{H}\setminus\{0\}$ we have
\[
\|f\|_{L^2(\Omega)}\le \mathfrak{c}\Psi \left( N_\beta(f)/\|\mathbf{I}(f)\|_{L^1((0,\mathfrak{t}))}\right)N_\beta(f) .
\]
\end{corollary}

\section{Proof of Theorem \ref{theoremin1}}

The proof of Theorem \ref{theoremin1} is based on an interior observability inequality. Let $\tilde{V}=V+\mathfrak{c}_1$ so as to satisfy $\lambda_{\tilde{V}}^j=\lambda_V^j+\mathfrak{c}_1\ge 0$ and $\phi_{\tilde{V}}^j=\phi_V^j$, $j\ge 1$. For simplicity, we use hereafter the notations $\lambda^j=\lambda_{\tilde{V}}^j$ and $\phi^j=\phi_V^j$, $j\ge 1$. Define for all $\lambda > 0$ the orthogonal projector $\Pi_\lambda$ as follows
\[
\Pi_\lambda u=\sum_{\lambda^j\le \lambda}(u|\phi^j)\phi^j,\quad u\in L^2(\Omega).
\]

We need the following preliminary result which can be seen as an observability inequality of a linear combination of eigenfunctions.

\begin{lemma}\label{lemmas1}
Let $u\in L^2(\Omega)$ and $\lambda \ge 0$. Then we have 
\begin{equation}\label{s1.6}
\mathbf{c}\|\Pi_\lambda u\|_{L^2(\Omega)}\le e^{2\sqrt{\lambda}}L_\rho(\Pi_\lambda u),\quad 0<\rho <\rho_0.
\end{equation}
\end{lemma}

\begin{proof}
Let $a_j:=(u|\phi^j)$, $j\ge 1$,  $v=\Pi_\lambda u$ and  
\[
w(x,y)=\sum_{\lambda^j\le \lambda} \frac{\sinh \left(\sqrt{\lambda ^j}\, y\right)}{\sqrt{\lambda^j}}a_j\phi^j(x),\quad (x,y)\in \Omega \times (0,1).
\]
It follows from Theorem \ref{gquc} of Appendix \ref{appA}, with $\Omega$ and $\omega$ replaced by $\Omega\times (0,1)$ and $\omega\times (1/4,3/4)$, respectively, 
\begin{equation}\label{s1.1}
\mathbf{c}\|w\|_{L^2(\Omega\times (0,1))} \le  H_\rho\|w\|_{L^2(\omega\times (0,1))}+\rho^s\|w\|_{H^1(\Omega\times (0,1))},\quad 0<\rho<\rho_0.
\end{equation} 
Since
\begin{align*}
\int_{-1/2}^{1/2}\int_\Omega |w(x,y)|^2&dxdy=\sum_{\lambda^j\le \lambda}|a_j|^2\int_{-1/2}^{1/2}\left| \frac{\sinh \left(\sqrt{\lambda^j}\, y\right)}{\sqrt{\lambda^j}}\right|^2dy
\\
&\ge \sum_{\lambda^j\le \lambda}|a_j|^2\int_{-1/2}^{1/2}y^2dy=c\sum_{\lambda^j\le \lambda}|a_j|^2=c\|v\|_{L^2(\Omega)}^2,
\end{align*}
for some universal constant $c>0$, \eqref{s1.1} implies
\[
\mathbf{c}\|v\|_{L^2(\Omega)}\le e^{2\sqrt{\lambda}}L_\rho (v),
\]
which is the expected inequality.
\end{proof}

We will also need the following technical lemma.

\begin{lemma}\label{tl}
Define for all $0<X\le Y$,  $\Phi_\rho=H_\rho X+\rho^s Y$, $0<\rho <\hat{\rho}$. Then
\begin{equation}\label{tl1}
\min_{0<\rho<\hat{\rho}} \Phi_\rho \le \mathbf{c} Y\Phi (Y/X).
\end{equation}
\end{lemma}

\begin{proof}
In the case $Y/X>\varphi(\hat{\rho})$, we find $0<\rho_1<\hat{\rho}$ satisfying $\varphi(\rho_1)=Y/X$. Therefore,
\[
\min_{0<\rho<\hat{\rho}} \Phi_\rho \le \varphi(\rho_1)\le (\hat{c}+s)Y(\ln \ln (Y/X))^{-s}.
\]
When $Y/X\le \varphi(\hat{\rho})$, we have 
\[
Y\le \varphi(\hat{\rho})X
\]
and hence
\[
\min_{0<\rho<\hat{\rho}} \Phi_{\rho}\le Y(H_{\hat{\rho}}+\hat{\rho}^s\varphi(\hat{\rho}))(Y/X)^{-1}.
\]
The expected inequality then follows from the combination of the two cases above.
\end{proof}

We are now ready to prove our observability inequality which is stated in the following theorem. Our proof is inspired by that of \cite[Theorem 8]{BM}.
  
\begin{theorem}\label{thmObh1}
For all $f\in \mathcal{H}$  we have
\begin{align}\label{0.0}
\|\mathbb{T}_{\mathfrak{t}} f\|_{L^2(\Omega)}\le \mathbf{c}_{\mathfrak{t}}\|\mathbf{I}(f)\|_{L^1((0,\mathfrak{t}))},
\end{align}
where $\mathbf{I}(f)$ is as in Theorem \ref{theoremin1}.
\end{theorem}

\begin{proof}
 We use in this proof  the notations $\mathbb{T}_t=\mathbb{T}_{\tilde{V}}(t)=e^{-\mathfrak{c}_1t}\mathbb{T}_{V}(t)$, $t\ge 0$.  
  
Let $f\in \mathcal{H}$. Then we have
\begin{equation}\label{e8}
\mathbb{T}_tf=\sum_{j\ge 1}e^{-\lambda^jt}(f|\phi^j)\phi^j,\quad t\ge 0.
\end{equation}
For $\lambda > 0$, we write $f=f_0+f_1$ with $f_0=\Pi_\lambda f$ and $f_1=f-f_0$. That is, we have
\[
f_0=\sum_{\lambda^j\le \lambda}(f|\phi^j)\phi^j,\quad f_1=\sum_{\lambda^j> \lambda}(f|\phi^j)\phi^j .
\]
Hence, $\mathbb{T}_tf=\mathbb{T}_tf_0+\mathbb{T}_tf_1$ and
\[
\mathbb{T}_tf_0=\sum_{\lambda^j\le \lambda}e^{-\lambda^jt}(f|\phi^j)\phi^j,\quad \mathbb{T}_tf_1=\sum_{\lambda^j> \lambda}e^{-\lambda^jt}(f|\phi^j)\phi^j.
\]

Let $0<\rho <\hat{\rho}$. From \eqref{s1.6}, we obtain
\begin{equation}\label{e9}
\|\mathbb{T}_tf_0\|_{L^2(\Omega)}\le \mathbf{c}e^{2\sqrt{\lambda}}\left(H_\rho\|\mathbb{T}_tf_0\|_{L^2(\omega)}+\rho^s\|\mathbb{T}_tf_0\|_{H^1(\Omega)}\right)
\end{equation}

On the other hand, we have for $0\le s< t$
\begin{align*}
\|\mathbb{T}_tf_1\|_{L^2(\omega)}^2\le  \|\mathbb{T}_tf_1\|_{L^2(\Omega)}^2&= \sum_{\lambda^j>\lambda} e^{-2\lambda^jt}|(f|\phi^j)|^2
\\
&=  \sum_{\lambda^j> \lambda} e^{-2\lambda^j(t-s)}e^{-2\lambda^j s}|(f|\phi^j)|^2
\\
&\le e^{-2\lambda (t-s)}\sum_{\lambda^j> \lambda} e^{-2\lambda^j s}|(f|\mathbf{e}_j)|^2
\\
&\le  e^{-2\lambda (t-s)}\|\mathbb{T}_sf\|_{L^2(\Omega)}^2
\end{align*}
and then
\begin{equation}\label{e10}
\|\mathbb{T}_tf_0\|_{L^2(\omega)}\le \|\mathbb{T}_tf\|_{L^2(\omega)}+e^{-\lambda (t-s)}\|\mathbb{T}_sf\|_{L^2(\Omega)}.
\end{equation}

Assume for the moment that
\begin{equation}\label{g1}
\nabla \mathbb{T}_tf_1= \sum_{\lambda^j> \lambda}e^{-\lambda^jt}(f|\phi^j)\nabla \phi^j.
\end{equation}
Therefore
\begin{align*}
\|\nabla \mathbb{T}_tf_1\|_{L^2(\Omega)}^2&= \sum_{\lambda^j> \lambda}\sum_{\lambda^k> \lambda}e^{-\lambda^jt}e^{-\lambda^kt}(f|\phi^j)(f|\phi^k)\int_\Omega \nabla \phi^j\cdot \nabla \phi^kdx
\\
&\le e^{-2\lambda (t-s)}\sum_{j\ge 1}\sum_{k\ge 1}|(f|\phi^j)||(f|\phi^k)|\left|\int_\Omega \nabla \phi^j\cdot \nabla \phi^kdx\right|.
\end{align*}
Using
\[
\int_\Omega \nabla \phi^j\cdot \nabla \phi^kdx=-\int_\Omega V\phi^j\phi^kdx+\lambda^j\int_\Omega \phi^j\phi^kdx,
\]
\cite[Lemma 1.1]{Ch19} and \cite[Subsection 2.3]{Ch24}, we obtain
\[
\left|\int_\Omega \nabla \phi^j\cdot \nabla \phi^kdx\right|\le \mathbf{c}\left[(1+\lambda^j)^{1/2}(1+\lambda^k)^{1/2}+\lambda^j\delta_{jk}\right].
\]
Therefore, we get
\[
\mathbf{c}e^{\lambda (t-s)}\|\nabla \mathbb{T}_tf_1\|_{L^2(\Omega)}\le \sum_{j\ge 1}(1+\lambda^j)^{1/2}|(f|\phi^j)|+\left(\sum_{j\ge 1}\lambda^j|(f|\phi^j)|^2\right)^{1/2}<\infty .
\]
In consequence, the series in the right hand side of \eqref{g1} converges in $L^2(\Omega)$ and we have
\begin{equation}\label{g2}
\|\mathbb{T}_tf_1\|_{H^1(\Omega)}\le \mathbf{c}e^{-\lambda (t-s)}N(f).
\end{equation}
By using \eqref{e10} and \eqref{g2} in \eqref{e9}, we obtain
\begin{align*}
&\|\mathbb{T}_tf\|_{L^2(\Omega)}\le \mathbf{c}e^{2\sqrt{\lambda}}\Big(H_\rho\left[\|\mathbb{T}_tf\|_{L^2(\omega)}+e^{-\lambda(t-s)}\|\mathbb{T}_sf\|_{L^2(\Omega)}\right]\ 
\\
&\hskip 5cm  \rho^s \left[\|\mathbb{T}_tf\|_{H^1(\Omega)} +e^{-\lambda(t-s)}N(f) \right]\Big),
\end{align*}
from which we deduce
\begin{align}
&\|\mathbb{T}_tf\|_{L^2(\Omega)}\le \mathbf{c}e^{2\sqrt{\lambda}}\left(\left[H_\rho\|\mathbb{T}_tf\|_{L^2(\omega)}+\rho^s N(f)\right]\right.\label{x1} 
\\
&\hskip 6cm \left. +e^{-\lambda(t-s)}H_\rho\|\mathbb{T}_sf\|_{L^2(\Omega)} \right).\nonumber
\end{align}

By applying \eqref{tl1} with $X=\|\mathbb{T}_tf\|_{L^2(\omega)}$ and $Y=N(f)$ (note that, since \eqref{g2} holds also when $f_1$ is replaced by $f$, we have $X\le Y$) and using that $\min_{0<\rho<\hat{\rho}}H_\rho\le H_{\hat{\rho}/2}$, we obtain
\begin{equation}\label{x2}
\|\mathbb{T}_tf\|_{L^2(\Omega)}\le \mathbf{c}e^{2\sqrt{\lambda}}\left( \mathbf{I}(f)(t)+ e^{-\lambda(t-s)}\|\mathbb{T}_sf\|_{L^2(\Omega)}\right). 
\end{equation}

Let $0<\epsilon <1$. Then \eqref{x2} yields
\begin{align*}
&\|\mathbb{T}_tf\|_{L^2(\Omega)} 
\\
&\hskip .5cm\le \mathbf{c}e^{2\sqrt{\lambda}-\epsilon(t-s)\lambda}\left(e^{\epsilon \lambda (t-s)}\mathbf{I}(f)(t)+e^{-(1-\epsilon)\lambda (t-s)}\|\mathbb{T}_sf\|_{L^2(\Omega)}\right).
\end{align*}
 As 
\[
\sup_{\lambda> 0} e^{2\sqrt{\lambda}-\epsilon(t-s)\lambda}=e^{ [\epsilon (t-s)]^{-1}}, 
\]
we find
\begin{align}
&\|\mathbb{T}_tf\|_{L^2(\Omega)} \label{e12}
\\
&\hskip .5cm\le \mathbf{c} e^{ [\epsilon (t-s)]^{-1}}\left(e^{\epsilon \lambda (t-s)}\mathbf{I}(f)(t)+e^{-(1-\epsilon)\lambda (t-s)}\|\mathbb{T}_sf\|_{L^2(\Omega)}\right).\nonumber
\end{align}

If $\|\mathbb{T}_sf\|_{L^2(\Omega)}>\mathbf{I}(f)(t)$, then 
\[
\lambda =\frac{1}{(t-s)}\ln \left(\|\mathbb{T}_sf\|_{L^2(\Omega)}/\mathbf{I}(f)(t)\right)
\]
in \eqref{e12} gives
\begin{equation}\label{e13}
\|\mathbb{T}_tf\|_{L^2(\Omega)}\le \mathbf{c} e^{[\epsilon (t-s)]^{-1}}\mathbf{I}(f)(t)^{1-\epsilon}\|\mathbb{T}_sf\|_{L^2(\Omega)}^\epsilon.
\end{equation}

When $\|\mathbb{T}_sf\|_{L^2(\Omega)}\le \mathbf{I}(f)(t)$, we have
\[
\|\mathbb{T}_tf\|_{L^2(\Omega)}\le \|\mathbb{T}_sf\|_{L^2(\Omega)}=  \|\mathbb{T}_sf\|_{L^2(\Omega)}^{1-\epsilon} \|\mathbb{T}_sf\|_{L^2(\Omega)}^\epsilon\le  \mathbf{I}(f)(t)^{1-\epsilon} \|\mathbb{T}_sf\|_{L^2(\Omega)}^\epsilon
\]
and therefore \eqref{e13} holds also in this case.

Let $0\le s <t \le \tau\le \mathfrak{t}$. As $\|\mathbb{T}_ \tau f\|_{L^2(\Omega)}\le \|\mathbb{T}_tf\|_{L^2(\Omega)}$,  by integrating \eqref{e13} over $(s+( \tau-s)/2, \tau)$, we find
\[
[( \tau-s)/2]\|\mathbb{T}_\tau f\|_{L^2(\Omega)}\le \mathbf{c} e^{ [\epsilon ( \tau-s)]^{-1}}\int_{s+( \tau-s)/2}^\rho \mathbf{I}(f)(t)^{1-\epsilon}\|\mathbb{T}_sf\|_{L^2(\Omega)}^\epsilon dt
\]
and therefore
\[
( \tau-s)\|\mathbb{T}_\tau f\|_{L^2(\Omega)}\le \mathbf{c}e^{ [\epsilon ( \tau-s)]^{-1}}\int_s^ \tau \mathbf{I}(f)(t)^{1-\epsilon}\|\mathbb{T}_sf\|_{L^2(\Omega)}^\epsilon dt.
\]
By applying H\"older's inequality, we get
\[
( \tau-s)\|\mathbb{T}_ \tau f\|_{L^2(\Omega)}\le \mathbf{c} e^{[\epsilon ( \tau-s)]^{-1}}\left(\int_s^\rho \mathbf{I}(f)(t)dt\right)^{1-\epsilon}(\tau-s)^\epsilon \|\mathbb{T}_sf\|_{L^2(\Omega)}^\epsilon 
\]
and then 
\begin{equation}\label{e15.1}
\|\mathbb{T}_ \tau f\|_{L^2(\Omega)}\le \mathbf{c}e^{[4\epsilon ( \tau-s)/5]^{-1}}\left(\int_s^\tau \mathbf{I}(f)(t)dt\right)^{1-\epsilon} \|\mathbb{T}_sf\|_{L^2(\Omega)}^\epsilon .
\end{equation}

Now,  \eqref{e15.1} implies for all $\alpha >0$
\begin{align*}
&e^{-\alpha [4\epsilon ( \tau-s)/5]^{-1}}\|\mathbb{T}_\tau f\|_{L^2(\Omega)}
\\
&\hskip 2cm \le \mathbf{c}e^{(1-\alpha) [4\epsilon ( \tau-s)/5]^{-1}}\left(\int_s^\tau \mathbf{I}(f)(t)dt\right)^{1-\epsilon}\|\mathbb{T}_sf\|_{L^2(\Omega)}^\epsilon.
\end{align*}
Thus,
\begin{align*}
&e^{-\alpha [4\epsilon (\tau-s)/5]^{-1}}\|\mathbb{T}_\tau f\|_{L^2(\Omega)}
\\
&\hskip .3cm \le \left[\mathbf{c} e^{-(\alpha/2) [4\epsilon (\tau-s)/5]^{-1}}\left(\int_s^\tau \mathbf{I}(f)(t)dt\right)^{1-\epsilon}\right]\left[e^{(1-\alpha/2) [4\epsilon (\tau-s)/5]^{-1}}\|\mathbb{T}_sf\|_{L^2(\Omega)}^\epsilon\right].
\end{align*}
This and the Young's inequality $\mathbf{x}\mathbf{y}\le (1-\epsilon)\mathbf{x}^{1/(1-\epsilon)}+\epsilon \mathbf{y}^{1/\epsilon}$ yield
\begin{align*}
&e^{-\alpha [4\epsilon (\tau-s)/5]^{-1}}\|\mathbb{T}_\tau f\|_{L^2(\Omega)}
 \le \mathbf{c}^{1/(1-\epsilon)}e^{-(\alpha/[2(1-\epsilon)]) [4\epsilon (\tau-s)/5]^{-1}}\int_s^\tau \mathbf{I}(f)(t)dt
 \\
&\hskip 6.9cm +e^{[(1-\alpha/2)/\epsilon] [4\epsilon (\tau-s)/5]^{-1}}\|\mathbb{T}_sf\|_{L^2(\Omega)}.
\end{align*}

Let $\mathbf{b}>0$ and $\mathbf{d}>0$. Then choose $\epsilon$ so that $2\epsilon<\mathbf{d}^{-1}$ and  fix $\alpha >0$ in such a way that $\alpha \ge 2\mathbf{b}$ and $(\alpha/2-1)/\epsilon>\alpha \mathbf{d}$. This choice allows us to derive from the last inequality
\begin{align*}
&e^{-\alpha [4\epsilon (\tau-s)/5]^{-1}}\|\mathbb{T}_\tau f\|_{L^2(\Omega)}\le \mathbf{c}^{1/(1-\epsilon)}e^{-\mathbf{b} [4\epsilon (\tau-s)/5]^{-1}}\int_s^\tau \mathbf{I}(f)(t)dt
\\
&\hskip 7cm  +e^{-\alpha \mathbf{d}[4\epsilon (\tau-s)/5]^{-1}}\|\mathbb{T}_sf\|_{L^2(\Omega)}.
\end{align*}
Equivalently, we have 
\begin{align}
&e^{-\alpha [4\epsilon (\tau-s)/5]^{-1}}\|\mathbb{T}_\tau f\|_{L^2(\Omega)}-e^{-\alpha \mathbf{d}[4\epsilon (\tau-s)/5]^{-1}}\|\mathbb{T}_sf\|_{L^2(\Omega)}\label{e14}
\\
&\hskip 5cm  \le \mathbf{c}^{1/(1-\epsilon)}e^{-\mathbf{b} [4\epsilon (\tau-s)/5]^{-1}}\int_s^\tau \mathbf{I}(f)(t)dt \nonumber
\end{align}

Let $(t_j)_{j\ge 0}$ be a decreasing sequence of $(0,\mathfrak{t}]$ satisfying $t_0=\mathfrak{t}$ and there exists $\theta \in (0,1)$ so that $(t_j-t_{j+1})\ge \theta (t_{j-1}-t_j)$ for every $j\ge 1$. By applying \eqref{e14} with $\tau=t_j$, $s=t_{j+1}$ and $\mathbf{d}=\theta^{-2}$, we find
\begin{align}
&e^{-\alpha [4\epsilon (t_j-t_{j+1})/5]^{-1}}\|\mathbb{T}_{t_j} f\|_{L^2(\Omega)}-e^{-\alpha [4\epsilon (t_{j+1}-t_{j+2})/5]^{-1}}\|\mathbb{T}_{t_{j+1}}f\|_{L^2(\Omega)}\label{e15}
\\
&\hskip 7cm  \le \mathbf{c}^{1/(1-\epsilon)}\int_{t_{j+1}}^{t_j} \mathbf{I}(f)(t)dt. \nonumber
\end{align}
Let $\ell \ge 0$ be an integer. By using
\begin{align*}
&\sum_{j=0}^\ell \left[e^{-\alpha [4\epsilon (t_j-t_{j+1})/5]^{-1}}\|\mathbb{T}_{t_j} f\|_{L^2(\Omega)}-e^{-\alpha [4\epsilon (t_{j+1}-t_{j+2})/5]^{-1}}\|\mathbb{T}_{t_{j+1}}f\|_{L^2(\Omega)}\right]
\\
&\hskip 3cm =e^{-\alpha [4\epsilon (\mathfrak{t}-t_1)/5]^{-1}}\|\mathbb{T}_{\mathfrak{t}} f\|_{L^2(\Omega)}-e^{-\alpha [4\epsilon (t_{\ell+1}-t_{\ell+2})/5]^{-1}}\|\mathbb{T}_{t_{\ell+1}}f\|_{L^2(\Omega)},
\end{align*}

\[
\lim_{\ell \rightarrow \infty}e^{-\alpha [4\epsilon (t_{\ell+1}-t_{\ell+2})/5]^{-1}}\|\mathbb{T}_{t_{\ell+1}}f\|_{L^2(\Omega)}=0
\]
(observe that $\|\mathbb{T}_{t_{\ell+1}}f\|_{L^2(\Omega)}\le \|f\|_{L^2(\Omega)}$) and 
\[
\sum_{j=0}^\ell\int_{t_{j+1}}^{t_j} \mathbf{I}(f)dt=\int_{t_{\ell+1}}^{\mathfrak{t}} \mathbf{I}(f)(t)dt \le  \|\mathbf{I}(f)\|_{L^1((0,\mathfrak{t}))},
\]
we deduce from \eqref{e15}
\[
\|\mathbb{T}_{\mathfrak{t}} f\|_{L^2(\Omega)}\le\mathbf{c}^{1/(1-\epsilon)}e^{\alpha (4\epsilon \mathfrak{t}/5)^{-1}}\|\mathbf{I}(f)\|_{L^1((0,\mathfrak{t}))}.
\]
This inequality with $\epsilon=1/8$, $\mathbf{d}=1/(4\epsilon)$, $\theta=2\sqrt{\epsilon}$ and $\alpha=5$ gives
\[
\|\mathbb{T}_{\mathfrak{t}} f\|_{L^2(\Omega)}\le \mathbf{c}e^{10\mathfrak{t}^{-1}}\|\Phi(f)\|_{L^1((0,\mathfrak{t}))}.
\]
The expected inequality then follows.
\end{proof}

\begin{proof}[Proof of Theorem \ref{theoremin1}]
Let $f\in \mathcal{H}_0$. From the formula
\[
\mathbb{T}_V(\mathfrak{t})f=\sum_{j\ge 1}e^{-\lambda_V^j\mathfrak{t}}(f|\phi_V^j)\phi_V^j,
\]
we obtain
\[
(f|\phi_V^j)=e^{\lambda_V^j\mathfrak{t}}(\mathbb{T}_V(\mathfrak{t})f|\phi_V^j),\quad j\ge 1.
\]
Let $\lambda >0$. Then we have
\begin{align*}
\sum_{\lambda_V^j\le \lambda }|(f|\phi_V^j)|^2&\le e^{2\lambda\mathfrak{t}}\sum_{\lambda_V^j\le \lambda }|(\mathbb{T}_V(\mathfrak{t})f|\phi_V^j)|^2
\\
&\le e^{2\lambda\mathfrak{t}}\| \mathbb{T}_V(\mathfrak{t})f\|_{L^2(\Omega)}^2.
\end{align*}
Combined with \eqref{0.0}, this inequality implies
\begin{equation}\label{Obs2}
\left(\sum_{\lambda_V^j\le \lambda }|(f|\phi_V^j)|^2\right)^{1/2} \le \mathbf{c}_{\mathfrak{t}}\|\mathbf{I}(f)\|_{L^1((0,\mathfrak{t}))}
\end{equation}

On the other hand, using 
\[
\sum_{\lambda_V^j\ge\lambda} |(f|\phi_V^j)|^2\le \lambda^{-2\beta}\sum_{\lambda_V^j\ge \lambda}(\lambda_V^j)^{2\beta}|(f|\phi_V^j)|^2,
\]
we obtain 
\begin{equation}\label{in3}
\sum_{\lambda_V^j\ge\lambda} |(f|\phi_V^j)|^2\le \lambda^{-2\beta}[f]_\beta^2 .
\end{equation}
Putting together \eqref{Obs2} and \eqref{in3}, we obtain
\[
\|f\|_{L^2(\Omega)} \le \mathbf{c}_{\mathfrak{t}}\|\mathbf{I}(f)\|_{L^1((0,\mathfrak{t}))}+\lambda^{-\beta}[f]_\beta .
\]
This is the expected inequality.
\end{proof}

\appendix

\section{Global quantitative unique continuation for the Sch\"odinger equation with unbounded potential}\label{appA}

Let $\Omega$ be a bounded Lipschitz domain of $\mathbb{R}^n$, $n\ge 3$, $p:=2n/(n+2)$ and $p':=2n/(n-2)$ its conjugate ($1/p+1/p'=1$). The following notations will be used throughout this appendix.
\begin{align*}
&\Omega^\delta=\{x\in \Omega;\; \mathrm{dist}(x,\partial \Omega)>\delta\},
\\
&\Omega_\delta=\{x\in \Omega;\; \mathrm{dist}(x,\partial \Omega)<\delta\}.
\end{align*}
We assume that there exits  $\delta_0>0$ such that $\Omega^\delta$ is connected for all $0<\delta <\delta_0$.

Recall that
\[
\mathscr{V}:=\{ V\in L^{n/2}(\Omega);\; \|V\|_{L^{n/2}(\Omega)}\le \vartheta\}.
\]
Here $\vartheta>0$ in chosen in such a way that $2\vartheta\sigma^2<1$, where
\[
\sigma:=\sup \{ \|w\|_{L^{p'}(\Omega)};\; w\in H_0^1(\Omega),\;  \|\nabla w\|_{L^2(\Omega)}=1\}.
\]

Our goal in this appendix is to prove the following global quantitative unique continuation for the operator $\Delta +V$, $V\in \mathscr{V}$. 

\begin{theorem}\label{thmgquc}
Let $0<s<1/2$, $\mathbf{D}=\mathrm{diam}(\Omega)$ and $\omega\Subset\Omega$. There exist $\mathbf{c}=\mathbf{c}(n,\Omega,\omega, \vartheta,s)$, $\rho_0=\rho_0(\omega,\mathbf{D})$ and $\hat{c}=\hat{c}(n,\mathbf{D},s)$ such that for all $V\in \mathscr{V}$ and $u\in W^{2,p}(\Omega)$ satisfying $(\Delta +V)u=0$ we have for all $0<\rho <\rho_0$
\begin{equation}\label{gquc}
\mathbf{c}\|u\|_{L^2(\Omega)} \le  e^{e^{\hat{c}\rho^{-n}}}\|u\|_{L^2(\omega)}+\rho^s\|u\|_{H^1(\Omega)}.
\end{equation}
\end{theorem}

The main ingredient we used to prove Theorem \ref{thmgquc} is an $L^p$-$L^{p'}$ Carleman inequality by Jerison and Kenig \cite{JK}. Similar result was already established by the author in \cite{Ch20} when $\Delta$ is replaced by an elliptic operator of the form $A=\sum_{ij}\partial_j(a_{ij}\partial_i \cdot)$. This result is based on a $L^2$ Carleman inequality and can be extended to $A+V$ with $V\in L^n(\Omega)$. 

When $\Omega$ is of class $C^{1,1}$, we can modify the proof of Theorem \ref{thmgquc} by using the same arguments as in the proof of \cite[Theorem 1.1]{Ch24_2} to obtain the following result.

\begin{theorem}\label{thmguc}
Assume that $\Omega$ is of class $C^{1,1}$. Let $0<s<1/2$ and $\zeta=(n,\Omega ,\vartheta, \omega,s)$. There exist $\mathbf{c}=\mathbf{c}(\zeta)$, $\mathbf{\hat{c}}=\mathbf{\hat{c}}(\zeta)$ and $\alpha =\alpha(\Omega)>2$ so that for all $V\in \mathscr{V}$ and $u\in W^{2,p}(\Omega)$ satisfying $(\Delta +V)u=0$ we have
\begin{equation}\label{guc}
 \mathbf{c}\|u\|_{L^2(\Omega)}\le  r^s \|u\|_{H^1(\Omega)}+e^{\mathbf{\hat{c}}/r^\alpha}\|u\|_{L^2(\omega)}.
\end{equation}
\end{theorem}

Before proving theorem \eqref{thmgquc}, we establish a Cacciopoli-type inequality and a three-ball inequality.

\subsection{Cacciopoli type inequality}

We will use henceforth that $W^{2,p}(\Omega)$ is continuously embedded in $H^1(\Omega)$ and therefore it is also continuously embedded in $L^{p'}(\Omega)$.

\begin{lemma}\label{lemma0}
Let $\omega_0\Subset \omega_1\Subset \Omega$ and $d=\mathrm{dist}(\omega_0,\partial \omega_1)$. There exists a constant $\mathbf{c}=\mathbf{c}(n,\Omega,\vartheta)$ such that for any $u\in W^{2,p}(\Omega)$ and $V\in \mathscr{V}$ we have
\begin{equation}\label{ca1}
\mathbf{c}\|\nabla u\|_{L^2(\omega_0)}\le \|(\Delta+V) u\|_{L^p(\Omega)}+d^{-1}\|u\|_{L^2(\omega_1)}.
\end{equation}
\end{lemma}

\begin{proof}
Pick $\chi \in C_0^\infty (\omega_1)$ satisfying  $0\le \chi \le 1$, $\chi =1$ in a neighborhood of $\omega_0$ and $|\partial^\alpha\chi |\le c_0d^{-1}$ for each $|\alpha|=1$, where $c_0>0$ is a universal constant.

Let $u\in W^{2,p}(\Omega)$ and $\epsilon>0$. In light of H\"older's inequality, we obtain
\[
\left| \int_\Omega \chi^2 u(\Delta+V) udx\right| \le \|(\Delta+V) u\|_{L^p(\Omega)}\|\chi u\|_{L^{p'}(\Omega)}
\]
and hence
\begin{align*}
\left| \int_\Omega \chi^2 u(\Delta+V) udx\right| &\le (2\epsilon)^{-1}\|(\Delta+V) u\|_{L^p(\Omega)}^2+2^{-1}\epsilon\|\chi u\|_{L^{p'}(\Omega)}^2
\\
&\le (2\epsilon)^{-1}\|(\Delta+V) u\|_{L^p(\Omega)}^2+2^{-1}\epsilon\sigma^2\|\nabla (\chi u)\|_{L^2(\omega)}^2.
\end{align*}
That is we have
\begin{align}
\left| \int_\Omega\chi^2 u(\Delta+V) udx\right| \le (2\epsilon)^{-1}&\|(\Delta+V) u\|_{L^p(\Omega)}^2\label{ca2}
\\
&+\epsilon\sigma^2\left(\|u\nabla \chi\|_{L^2(\Omega)}^2+\|\chi \nabla u\|_{L^2(\Omega)}^2\right).\nonumber
\end{align}

On the other hand,   integration by parts gives
\[
\int_\Omega\chi^2 u\Delta udx=-\int_\Omega\chi^2|\nabla u|^2dx-\int_\Omega u\nabla \chi^2 \cdot \nabla udx.
\]
Thus, 
\[
\|\chi \nabla u\|_{L^2(\Omega)}^2\le \left|\int_\Omega\chi^2 u\Delta udx\right|+2\left|\int_\Omega u\chi \nabla \chi \cdot \nabla udx\right|.
\]
Applying Cauchy-Schwarz's inequality and then a convexity inequality to the second term of the right hand side of inequality above, we obtain
\[
\|\chi \nabla u\|_{L^2(\Omega)}^2\le \left|\int_\Omega \chi^2 u\Delta udx\right|+\epsilon^{-1}\|u\nabla \chi\|_{L^2(\Omega)}^2+\epsilon\|\chi \nabla u\|_{L^2(\Omega)}^2.
\]
Whence, 
\[
(1-\epsilon)\|\chi \nabla u\|_{L^2(\Omega)}^2\le \left|\int_\Omega\chi^2 u\Delta udx\right|+\epsilon^{-1}\|u\nabla \chi\|_{L^2(\Omega)}^2,
\]
from which we obtain
\begin{equation}\label{ca3}
(1-\epsilon)\|\chi \nabla u\|_{L^2(\Omega)}^2\le \left|\int_\Omega \chi^2 u(\Delta+V) udx\right|+ \left|\int_\Omega V\chi^2 u^2dx\right|+\epsilon^{-1}\|u\nabla \chi\|_{L^2(\Omega)}^2.
\end{equation}
Next, applying twice H\"older's inequality, we obtain
\[
\left|\int_\Omega q\chi^2 u^2dx\right|\le \|V\chi u\|_{L^p(\Omega)}\|\chi u\|_{L^{p'}(\Omega)}\le \|V\|_{L^{n/2}(\Omega)}\|\chi u\|_{L^{p'}(\Omega)}^2.
\]
Thus, we have
\begin{align*}
\left|\int_\Omega V\chi^2 u^2dx\right|&\le  \|V\|_{L^{n/2}(\Omega)}\sigma^2\|\nabla (\chi u)\|_{L^2(\Omega)}^2
\\
& \le  2\|V\|_{L^{n/2}(\Omega)}\sigma^2\left(\|u\nabla \chi \|_{L^2(\Omega)}^2+\|\chi \nabla u\|_{L^2(\Omega)}^2\right)
\\
& \le  2\vartheta\sigma^2\left(\|u\nabla \chi \|_{L^2(\Omega)}^2+\|\chi \nabla u\|_{L^2(\Omega)}^2\right).
\end{align*}
This inequality in \eqref{ca3} gives
\begin{equation}\label{ca4}
(1-\epsilon-2\vartheta\sigma^2)\|\chi \nabla u\|_{L^2(\Omega)}^2\le \left|\int_\Omega\chi^2 u(\Delta+q) udx\right|+(\epsilon^{-1}+2\vartheta\sigma^2)\|u\nabla \chi\|_{L^2(\Omega)}^2.
\end{equation}
By putting \eqref{ca2} in \eqref{ca4}, we end up getting
\begin{align*}
(1-2\vartheta\sigma^2-\epsilon(1+\sigma^2))\|\chi \nabla u\|_{L^2(\Omega)}^2 &\le (2\epsilon)^{-1}\|(\Delta+V) u\|_{L^p(\Omega)}^2
\\
&\qquad +(\epsilon^{-1}+\epsilon\sigma^2+2\vartheta\sigma^2)\|u\nabla \chi\|_{L^2(\Omega)}^2.
\end{align*}
The expected inequality follows from the choice in the above inequality $\epsilon=(1-2\vartheta\sigma^2)/(2+2\sigma^2)$.
\end{proof}

\subsection{Three-ball inequality}

\begin{proposition}\label{pro3bi}
There exists a constant $\mathbf{c}=\mathbf{c}(n,\Omega,\vartheta)$ so that for any $V\in \mathscr{V}$, $u\in W^{2,p}(\Omega)$ satisfying $(\Delta+V)u=0$, $x_0\in \Omega^{4\delta}$ and $0<r\le \delta$ we have 
\begin{equation}\label{3bi}
\mathbf{c}r^2\|u\|_{L^2(B(x_0,2r)}\le \|u\|_{L^2(B(x_0,3r))}^{\varsigma}\|u\|_{L^2(B(x_0,r))}^{1-\varsigma},
\end{equation}
where $\varsigma=(\ln 16-\ln 5)/(\ln 18-\ln 5)$.
\end{proposition}

Due to translation invariance, it suffices to prove Proposition \ref{pro3bi} with $x_0=0$. For this purpose, the open ball with center $0$ and radius $r>0$ will be denoted $B_r$ and $B:=B_{4r}$.

Define
\[
\Lambda=\{ \lambda >0;\; \mathrm{dist}(\lambda ,\mathbb{N}+(n-2)/2)\ge 1/2\}.
\]

The following Carleman inequality can be proved similarly as \cite[Lemma 2.2]{Ch23}, where $\dot{B}=B\setminus \{0\}$.

\begin{lemma}\label{lemma1}
There exists $\kappa=\kappa(n,\vartheta)>0$ such that for all  $u\in W^{2,p}(B)$ with $\mathrm{supp}(u)\subset \dot{B}$,  $q\in \mathscr{V}$ and $\lambda \in \Lambda$ we have
\begin{equation}\label{a1}
\||x|^{-\lambda}u\|_{L^{p'}(B)}\le \kappa \||x|^{-\lambda}(\Delta +V)u\|_{L^p(B)}.
\end{equation}
\end{lemma}

\begin{proof}[Proof of Proposition \ref{pro3bi}]
In this proof $\mathbf{c}=\mathbb{c}(n,\Omega,\vartheta)>0$ is a generic constant and 
\[
[\![a,b]\!]:=\{x\in \mathbb{R}^n;\, a<|x]<b\},\quad 0\le a<b.
\]
Let $(r_j)_{1\le j\le 8}$ be an increasing sequence of $(0,4r)$ and $\chi\in C_0^\infty (B)$ satisfying $0\le \chi \le 1$,
\[
\chi=\left\{
\begin{array}{lll}
0&\quad \mbox{in}\; [\![0,r_1]\!]\cup [\![r_7,4]\!],
\\
1 &\mbox{in}\; [\![r_3,r_6]\!],
\end{array}
\right.
\]
and
\begin{align*}
&|\Delta \chi|+|\nabla \chi|^2\le \mathfrak{c}d_1^{-2}\quad \mbox{in}\; [\![r_2,r_3]\!],
\\
&|\Delta \chi|+|\nabla \chi|^2\le \mathfrak{c}d_2^{-2}\quad \mbox{in}\; [\![r_6,r_7]\!],
\end{align*}
where $\mathfrak{c}$ is a universal constant, $d_1=r_3-r_2$ and $d_2=r_7-r_6$.

Let $V\in \mathscr{V}$, $u\in W^{2,p}(\Omega)$ satisfying $(\Delta +q)u=0$ and $\lambda \in \Lambda$.  Then
\[
\||x|^{-\lambda}(\Delta +V)(\chi u)\|_{L^p(B)} = \||x|^{-\lambda}(\Delta \chi u+2\nabla \chi\cdot \nabla u)\|_{L^p(B)}
\]
and therefore 
\begin{align*}
&\mathbf{c}\||x|^{-\lambda}(\Delta +q)(\chi u)\|_{L^p(B)}\le  d_1^{-2}\||x|^{-\lambda}u\|_{L^p([\![r_2,r_3]\!])}+d_1^{-1}\||x|^{-\lambda}\nabla u\|_{L^p([\![r_2,r_3]\!])}
\\
&\hskip 4.5cm +d_2^{-2}\||x|^{-\lambda}u\|_{L^p([\![r_6,r_7]\!])}+d_2^{-1}\||x|^{-\lambda}\nabla u\|_{L^p([\![r_6,r_7]\!])}.
\end{align*}
In consequence, we obtain
\begin{align}
&\mathbf{c}\||x|^{-\lambda}(\Delta +V)(\chi u)\|_{L^p(B)}\label{a2}
\\
&\hskip 2cm\le  d_1^{-2}r_2^{-\lambda}\|u\|_{L^p([\![r_2,r_3]\!])}+d_1^{-1}r_2^{-\lambda}\|\nabla u\|_{L^p([\![r_2,r_3]\!])}\nonumber
\\
&\hskip 4cm +d_2^{-2}r_6^{-\lambda}\|u\|_{L^p([\![r_6,r_7]\!])}+d_2^{-1}r_6^{-\lambda}\|\nabla u\|_{L^p([\![r_6,r_7]\!])}.\nonumber
\end{align}
Assume that $r_2-r_1=r_4-r_3=d_1$ and $r_6-r_5=r_8-r_7=d_2$. In light of \eqref{ca1}, we obtain from \eqref{a2}
\begin{equation}\label{a3}
\mathbf{c}\||x|^{-\lambda}(\Delta +V)(\chi u)\|_{L^p(B)}\le  d_1^{-2}r_2^{-\lambda}\|u\|_{L^p([\![r_1,r_4]\!])}+ d_2^{-2}r_6^{-\lambda}\|u\|_{L^p([\![r_5,r_8]\!])}.
\end{equation}

On the hand, we have from \eqref{a1}
\[
\||x|^{-\lambda}\chi u\|_{L^{p'}(B)}\le \kappa \||x|^{-\lambda}(\Delta +V)(\chi u)\|_{L^p(B)}.
\]
This and \eqref{a3} imply
\[
\mathbf{c}r_5^{-\lambda}\|u\|_{L^2([\![r_3,r_5]\!])}\le  d_1^{-2}r_2^{-\lambda}\|u\|_{L^2([\![r_1,r_4]\!])}+ d_2^{-2}r_6^{-\lambda}\|u\|_{L^2([\![r_5,r_8]\!])}.
\]
Thus,
\begin{equation}\label{a4}
\mathbf{c}\|u\|_{L^2([\![r_3,r_5]\!])}\le  d_1^{-2}r_2^{-\lambda}r_5^\lambda\|u\|_{L^2([\![r_1,r_4]\!])}+ d_2^{-2}r_5^\lambda r_6^{-\lambda}\|u\|_{L^2([\![r_5,r_8]\!])}.
\end{equation}

Next, let us specify the sequence $(r_j)$. We choose $r_1=3r/8$, $r_2=5r/8$, $r_3=3r/4$, $r_4=1$, $r_5=2$, $r_6=9/4$,  $r_7=11/4$ and $r_8=3$. In this case $d_1=d_2=r/4$. This choice in \eqref{a4} gives
\[
\mathbf{c}r^2\|u\|_{L^2([\![3r/4,2r]\!])}\le  (5/16)^{-\lambda}\|u\|_{L^2([\![3r/8,r]\!])}+ (8/9)^\lambda\|u\|_{L^2([\![2r,3r]\!])}
\]
and therefore
\[
\mathbf{c}r^2\|u\|_{L^2(B_{2r})}\le  (5/16)^{-\lambda}\|u\|_{L^2(B_r)}+\mathbf{c} \|u\|_{L^2(B_{3r/4})}+ (8/9)^\lambda\|u\|_{L^2(B_{3r})}.
\]
In consequence, there exists $\lambda_\ast=\lambda_\ast (n,\vartheta)>0$ so that for all $\lambda \ge \lambda_\ast$ we have
\begin{equation}\label{a5}
\mathbf{c}r^2\|u\|_{L^2(B_{2r})}\le  (5/16)^{-\lambda}\|u\|_{L^2(B_r)}+ (8/9)^\lambda\|u\|_{L^2(B_{3r})}.
\end{equation}

From the uniqueness of continuation, if $u$ is non identically equal to zero, then $\|u\|_{L^2(B_{3r})}/\|u\|_{L^2(B_r)}>1$. Let $s=\ln(18/5)$ and $\tilde{\lambda}=s^{-1}\ln (\|u\|_{L^2(B_{3r})}/\|u\|_{L^2(B_r)})$. When $\tilde{\lambda}\le \lambda_\ast+1$, we have
\[
\|u\|_{L^2(B_{3r})}\le e^{s(\lambda_\ast+1)}\|u\|_{L^2(B_r)}
\]
and hence
\begin{equation}\label{a6}
\mathbf{c}r^2\|u\|_{L^2(B_{2r})}\le e^{s(1-\eta) (\lambda_\ast+1)}\|u\|_{L^2(B_{3r})}^\eta\|u\|_{L^2(B_r)}^{1-\eta},\quad 0\le \eta \le 1.
\end{equation}

Next, assume that $\tilde{\lambda}>\lambda_\ast+1$. If $\tilde{\lambda}\in \Lambda$, then $\lambda=\tilde{\lambda}$ in \eqref{a5} yields
\begin{equation}\label{a7}
\mathbf{c}r^2\|u\|_{L^2(B_{2r})}\le \|u\|_{L^2(B_{3r})}^{\varsigma}\|u\|_{L^2(B_r)}^{1-\varsigma}.
\end{equation}
In the case $\tilde{\lambda}\not\in \Lambda$, we find a positive integer $j$ so that $\tilde{\lambda}\in [j+(n-2)/2, j+1/2+(n-2)/2)$ or $\tilde{\lambda}\in (j-1/2+(n-2)/2,[j+(n-2)/2)]$. If $\tilde{\lambda}\in [j+(n-2)/2, j+1/2+(n-2)/2)$ we take $\lambda=j+1/2+(n-2)/2$ in \eqref{a5}. As $\tilde{\lambda}<\lambda<\tilde{\lambda}+1$, we verify that \eqref{a7} holds with the same $\varsigma$. When $\tilde{\lambda}\in (j-1/2+(n-2)/2), j+(n-2)/2]$, we proceed similarly as in the preceding case. By taking $\lambda= j-1/2+(n-2)/2$, we show that \eqref{a7} holds again with the same $\varsigma$.
\end{proof}

\subsection{Proof of Theorem \ref{thmgquc}}

The proof is decomposed into two steps. 

\subsubsection*{First step} In this step, $\mathbf{c}=\mathbf{c}(n,\Omega,\vartheta)>0$ will denote a generic constant.  Let
 \[
\eta_\delta:=\mathbf{c}_0e^{-\mathbf{c}_1\delta^{-n}},
\]
where
\[
\mathbf{c}_0=\varsigma e^{-2^{n-1/2}|\ln \varsigma |},\quad \mathbf{c}_1=2^{n/2-1/2}|\ln \varsigma |\mathbf{D}^n.
\]
Here $\varsigma$ is as Proposition \ref{pro3bi} and $\mathbf{D}=\mathrm{diam}(\Omega)$.

Let $Q$ be the smallest closed cube containing $\overline{\Omega}$. Then we have $|Q|=\mathbf{D}^n$. Let $0<\delta<\delta_0/4$. We divide $Q$ into $([\mathbf{D}/(\sqrt{n}\delta)]+1)^n:=m_{\delta}$ closed sub-cubes, where $[\mathbf{D}/(\sqrt{n}\delta)]$ is the integer part of $\mathbf{D}/(\sqrt{n}\delta)$. Let $(Q_j)_{1\le j\le m_\delta}$ be the family of these cubes. Note that , for each $j$, $|Q_j|< (\sqrt{n}\delta)^n$ and therefore $Q_j$ is contained in a ball $B_j$ of radius $\delta$. Define
\[
I_\delta=\{j \in \{1,\ldots,m_\delta\};\; Q_j\cap \overline{\Omega^{4\delta}}\ne \emptyset\}\quad \mathrm{and}\quad Q^\delta=\cup_{j\in I_\alpha}Q_j.
\]
In particular, $Q^\delta\subset \Omega^{3\delta}$ and since $\overline{\Omega^{4\delta}}$ is connected then so is $Q^\delta$. Let $x,y\in \Omega^{4\delta}$ and $\psi:[0,1]\rightarrow Q^\delta$ be a path joining $x$ to $y$ such that $\psi$ is constant on each $Q_j$ in such a way that the length of $\psi$ restricted to each $Q_j$ is smaller than $\sqrt{n}\delta$. Therefore, the length of $\psi$, denoted hereafter by $\ell(\psi)$ does not exceed $\sqrt{n}\delta m_\delta$.
Let $t_0=0$ and define the sequence $(t_k)$ as follows
\[
t_{k+1}=\inf\{t\in [t_k,1];\; \psi(t)\not\in B(\psi(t_k),\delta)\},\quad k\ge 0.
\]
Then $|\psi(t_{k+1})-\psi(t_k)|=\delta$. Thus, there exists a positive integer $p_\delta$ so that $\psi(1)\in B(\psi(t_{p_\delta}),\delta)$. As $\delta p_\delta \le \ell(\psi)\le \sqrt{n}\delta m_\delta$, we obtain $p_\delta\le \sqrt{n}m_\delta$.

Set $x_j=\psi(t_k)$, $j=0, \ldots,p_\delta$ and $x_{p_\delta+1}=y$. We verify that $B(x_j,3\delta)\subset \Omega$, $j=0, \ldots,p_\delta+1$, and  $B(x_{j+1},\delta)\subset B(x_j,2\delta)$, $j=0, \ldots,p_\delta$.

Let $V\in \mathscr{V}$ and $u\in W^{2,p}(\Omega)$ satisfying $(\Delta +V)u=0$ and $\|u\|_{L^2(\Omega)}=1$.
It follows from \eqref{3bi}
\[
\|u\|_{L^2(B(z,2\delta))}\le \mathbf{c}\delta^{-2}\|u\|_{L^2(B(z,3\delta))}^{1-\varsigma}\|u\|_{L^2(B(z,\delta))}^\varsigma,\quad z\in \Omega^{3\delta}, 
\]
Thus,
\begin{equation}\label{tbii}
\|u\|_{L^2(B(z,2\delta))}\le \mathbf{c}\|u\|_{L^2(B(z,\delta))}^\varsigma,\quad z\in \Omega^{3\delta}.
\end{equation}
Inequality \eqref{tbii} with $x=x_j$, $j=0, \ldots p_\delta$, yields
\begin{equation}\label{tbii_1}
\|u\|_{L^2(B(x_{j+1},\delta))}\le \mathbf{c}\delta^{-2}\|u\|_{L^2(B(x_j,\delta))}^\varsigma,\quad j=0, \ldots p_\delta.
\end{equation}
By induction in $j$, we obtain from \eqref{tbii_1}
\begin{equation}\label{tbii_2}
\|u\|_{L^2(B(y,\delta))}\le [\mathbf{c}\delta^{-2}]^{1+\varsigma+\ldots +\varsigma^{p_\delta}}\|u\|_{L^2(B(x,\delta))}^{\varsigma^{p_\delta+1}}.
\end{equation}
By replacing if necessary $\mathbf{c}$, we suppose that $\mathbf{c}\delta^{-2}\ge 1$. Therefore, we obtain from \eqref{tbii_2} 
\begin{equation}\label{tbii_3}
\|u\|_{L^2(B(y,\delta))}\le \mathbf{c}\delta^{-2/(1-\varsigma)}\|u\|_{L^2(B(x,\delta))}^{\gamma_\delta},
\end{equation}
where 
\[
\gamma_\delta=\varsigma^{p_\delta+1}=\varsigma e^{-p_\delta |\ln \varsigma|}\ge \varsigma e^{-\sqrt{2}|\ln \varsigma|(\mathbf{D}/(\sqrt{2}\delta)+1)^n}.
\]
Using  
\[
(\mathbf{D}/(\sqrt{2}\delta)+1)^n\le 2^{n-1}\left[(\mathbf{D}/(\sqrt{2}\delta)^n+1\right]=2^{n/2-1}\mathbf{D}^n \delta^{-n}+2^{n-1},
\]
we find $\gamma_\delta\ge \eta_\delta$ and, as $\|u\|_{L^2(B(x,\delta))}\le 1$, we obtain
 \begin{equation}\label{c1}
\|u\|_{L^2(B(y,\delta))}\le \mathbf{c}\delta^{-2/(1-\varsigma)}\|u\|_{L^2(B(x,\delta))}^{\eta_\delta}.
\end{equation}

\subsubsection{Second step} In the following, $\mathbf{c}=\mathbf{c}(n,\Omega,\omega, \vartheta,s)$ is a generic constant. Let $0<\rho \le \delta_0/4$, $0<r<\rho/2$ and $\chi \in C_0^\infty (B_\rho)$ satisfying $0\le \chi \le 1$ and
\[
\chi=\left\{
\begin{array}{lll}
0&\quad \mbox{in}\; [\![0,r/2]\!]\cup [\![2\rho/3,\rho]\!],
\\
1 &\mbox{in}\; [\![r,\rho/2]\!],
\end{array}
\right.
\]
and
\begin{align*}
&|\Delta \chi|+|\nabla \chi|^2\le \mathfrak{c}r^{-2}\quad \mbox{in}\; [\![r/2,r]\!],
\\
&|\Delta \chi|+|\nabla \chi|^2\le \mathfrak{c}\rho^{-2}\quad \mbox{in}\; [\![\rho/2,2\rho/3]\!],
\end{align*}
where $\mathfrak{c}$ is a universal constant.

Let $V\in \mathscr{V}$, $y\in \Omega^{4\rho}$ and $u\in W^{2,p}(\Omega)$ satisfying $(\Delta +q)u=0$ and $\|u\|_{L^2(\Omega)}=1$. Upon making a translation,  we assume that $x=0$.  Then
\[
\||x|^{-\lambda}(\Delta +V)(\chi u)\|_{L^p(B_\rho)} = \||x|^{-\lambda}(\Delta \chi u+2\nabla \chi\cdot \nabla u)\|_{L^p(B_\rho)}
\]
and therefore 
\begin{align*}
&\mathbf{c}\||x|^{-\lambda}(\Delta +q)(\chi u)\|_{L^p(B_\rho)}
\\
&\hskip 1cm \le  r^{-2}\||x|^{-\lambda}u\|_{L^p([\![r/2,r]\!])}+r^{-1}\||x|^{-\lambda}\nabla u\|_{L^p([\![r/2,r]\!])}
\\
&\hskip3cm +\rho^{-2}\||x|^{-\lambda}u\|_{L^p([\![\rho/2,2\rho/3]\!])}+\rho^{-1}\||x|^{-\lambda}\nabla u\|_{L^p([\![\rho/2,2\rho/3]\!])}.
\end{align*}
We deduce from this inequality
\begin{align*}
&\mathbf{c}\||x|^{-\lambda}(\Delta +q)(\chi u)\|_{L^p(B_\rho)} \le  r^{-2}[r/2]^{-\lambda}\|u\|_{L^2([\![r/2,r]\!])}
\\
&\hskip 2cm +r^{-1}[r/2]^{-\lambda}\|\nabla u\|_{L^2([\![r(1-1/\lambda),r]\!])}
\\
&\hskip3cm +\rho^{-2}[\rho/2]^{-\lambda}\|u\|_{L^2([\![\rho/2,2\rho/3]\!])}+\rho^{-1}[\rho/2]^{-\lambda}\|\nabla u\|_{L^2([\![\rho/2,2\rho/3]\!])}.
\end{align*}
By using Caccioppoli's inequality, we obtain
\begin{align*}
&\mathbf{c}\||x|^{-\lambda}(\Delta +q)(\chi u)\|_{L^2(B_\rho)} \le  r^{-2}[r/2]^{-\lambda}\|u\|_{L^2([\![r/4,2r]\!])}
\\
&\hskip 6.5cm+\rho^{-2}[\rho/2]^{-\lambda}\|u\|_{L^2([\![3\rho/8,3\rho/4]\!])}.
\end{align*}

Next, suppose that $r< \rho/4$. Then \eqref{a1} implies
\begin{align*}
&\mathbf{c}[\rho/4]^{-\lambda}\|u\|_{L^2([\![r,\rho/4]\!])} \le  r^{-2}[r/2]^{-\lambda}\|u\|_{L^2([\![r/4,2r]\!])}
\\
&\hskip 6.5cm+\rho^{-2}[\rho/2]^{-\lambda}\|u\|_{L^2([\![3\rho/8,3\rho/4]\!])}.
\end{align*}
Hence,
\begin{equation}\label{w1}
\mathbf{c}\|u\|_{L^2([\![r,\rho/4]\!])} \le  r^{-2}[\rho/(2r)]^{\lambda}\|u\|_{L^2([\![r/4,2r]\!])}
+\rho^{-2}2^{-\lambda}\|u\|_{L^2([\![3\rho/8,3\rho/4]\!])}.
\end{equation}

Let
\[
\Phi(t)=r^{-2}[\rho/(2r)]^t\|u\|_{L^2([B_{2r})}+2^{-t}\rho^{-2}\|u\|_{L^2([\![3\rho/8,3\rho/4]\!])},\quad t>0.
\]
If $t\not\in \Lambda$, then there exists $j$ a positive integer so that $t\in (-1/2+j+(n-2)/2,1/2+j+(n-2)/2)$. That is we have $\lambda_0<t<\lambda_0+1$, where $\lambda_0=-1/2+j+(n-2)/2\in \Lambda$. In this case,
\[
\Phi (\lambda_0)\le 2\Phi (t).
\]
In consequence, \eqref{w1} yields
\begin{equation}\label{w2}
\mathbf{c}\|u\|_{L^2(B_{\rho/4})} \le  r^{-2}[\rho/(2r)]^t\|u\|_{L^2(B_{2r})}+2^{-t}\rho^{-2}\|u\|_{L^2([\![3\rho/8,3\rho/4]\!])},\quad t> 0.
\end{equation}
Let $x\in \Omega^{4\delta}$. Then \eqref{w2} and \eqref{c1} give for $0< r<\rho/4$ and $t>0$
\[
\mathbf{c}\|u\|_{L^2(B(y,\rho/4))} \le  r^{-2}[\rho/(2r)]^t[(2r)^{-2/(1-\varsigma)}\|u\|_{L^2(B(x,2r))}^{\eta_r}]+2^{-t}\rho^{-2}.
\]
This inequality yields
\[
\mathbf{c}\|u\|_{L^2(B(\Omega^{4\rho})} \le  \rho^{-n}r^{-2}\left[\rho/(2r)]^t[(2r)^{-2/(1-\varsigma)}\|u\|_{L^2(B(x,2r))}^{\eta_r}\right]+2^{-t}\rho^{-(2+n)}.
\]
When $x\in \omega$, reducing $\delta_0$ if necessary, we assume that $B(x,2r)\subset \omega$. Then we get
\[
\mathbf{c}\|u\|_{L^2(B(\Omega^{4\rho})} \le  \rho^{-n}r^{-2}\left[\rho/(2r)]^t[(2r)^{-2/(1-\varsigma)}\|u\|_{L^2(\omega)}^{\eta_r}\right]+2^{-t}\rho^{-(2+n)},
\]
where we used the fact that $\Omega^{4\delta}$ can be covered by $O(\rho^{-n})$ balls of radius $\rho/4$. 

Without the condition $\|u\|_{L^2(\Omega)}=1$, the previous inequality has to be replaced by
\begin{align*}
&\mathbf{c}\|u\|_{L^2(\Omega^{4\rho})} \le  \rho^{-n}r^{-2}\left[[\rho/(2r)]^t[(2r)^{-2/(1-\varsigma)}\|u\|_{L^2(\omega)}^{\eta_r}\|u\|_{L^2(\Omega)}^{1-\eta_r}\right]
\\
&\hskip 7.5cm +2^{-t}\rho^{-(2+n)}\|u\|_{L^2(\Omega)}.
\end{align*}
We combine this inequality with Hardy's inequality in order to obtain
\begin{align*}
&\mathbf{c}\|u\|_{L^2(\Omega)} \le  \rho^{-n}r^{-2}\left[[\rho/(2r)]^t[(2r)^{-2/(1-\varsigma)}\|u\|_{L^2(\omega)}^{\eta_r}\|u\|_{L^2(\Omega)}^{1-\eta_r}\right]
\\
&\hskip 6cm +2^{-t}\rho^{-(2+n)}\|u\|_{L^2(\Omega)}+\rho^s\|u\|_{H^1(\Omega)}.
\end{align*}
In this inequality we take $r=\rho/8$. We get
\begin{align*}
&\mathbf{c}\|u\|_{L^2(\Omega)} \le  \rho^{-(n+2)}\left[4^{t+2/(1-\varsigma)}\rho^{-2/(1-\varsigma)}\|u\|_{L^2(\omega)}^{\eta_{\rho/8}}\|u\|_{L^2(\Omega)}^{1-\eta_{\rho/8}}\right]
\\
&\hskip 6cm +2^{-t}\rho^{-(2+n)}\|u\|_{L^2(\Omega)}+\rho^s\|u\|_{H^1(\Omega)}.
\end{align*}
Next, we choose $t$ so that $2^{-t}\rho^{-(2+n)}= \rho^s$. That is, $t=\ln(\rho^{-(2+n+s)})/\ln 2$.  By this choice we obtain
\begin{equation}\label{w3}
\mathbf{c}\|u\|_{L^2(\Omega)} \le  \rho^{-\nu }\|u\|_{L^2(\omega)}^{\eta_{\rho/8}}\|u\|_{L^2(\Omega)}^{1-\eta_{\rho/8}}
+\rho^s\|u\|_{H^1(\Omega)},
\end{equation}
where $\nu=n+2+(n+2+s)\ln 4/\ln 2+2/(1-\varsigma)$.

On the other hand, it follows from Young's inequality that for any $\epsilon >0$, we have
\[
\|u\|_{L^2(\omega)}^{\eta_{\rho/8}}\|u\|_{L^2(\Omega)}^{1-\eta_{\rho/8}}\le \epsilon^{-1/\eta_{\rho/8}}\|u\|_{L^2(\omega)}+\epsilon^{1/(1-\eta_{\rho/8})}\|u\|_{L^2(\Omega)}.
\]
This inequality with $\epsilon= \rho^{(s+\nu)(1-\eta_{\rho/8})}$ and \eqref{w3} give
\begin{equation}\label{w4}
\mathbf{c}\|u\|_{L^2(\Omega)} \le  \rho^{-\nu }\rho^{-(s+\nu)(1-\eta_{\rho/8})/\eta_{\rho/8}}\|u\|_{L^2(\omega)}
+\rho^s\|u\|_{H^1(\Omega)},
\end{equation}
We verify that \eqref{w4} yields
\[
\mathbf{c}\|u\|_{L^2(\Omega)} \le  e^{e^{\hat{c}\rho^{-n}}}\|u\|_{L^2(\omega)}+\rho^s\|u\|_{H^1(\Omega)},
\]
where $\hat{c}=\hat{c}(n,\mathbf{d},s)>0$ is a constant. The proof of Theorem \ref{thmgquc} is then complete.


\begin{thebibliography}{99}


\bibitem{BM} Burq, N. ; Moyano, I. : Propagation of smallness and control for heat equations. J. Eur. Math. Soc. (JEMS) 25 (4) (2023), 1349-1377.

\bibitem{CK} Canuto, B. ; Kavian O. : Determining coefficients in a class of heat equations via boundary measurements. SIAM J. Math. Anal. 32 (2001), no. 5, 963-986.

\bibitem{Ch24} Choulli, M. :  A quantitative Borg-Levinson theorem for a large class of unbounded potentials. arXiv:2410.01346. 

\bibitem{Ch24_2} Choulli, M. : Uniqueness of continuation for semilinear elliptic equations. Partial Differ. Equ. Appl. 5, 22 (2024).

\bibitem{Ch23} Choulli, M.  : Quantitative strong unique continuation property for the Schrödinger operator with unbounded potential. arXiv:2309.0851.

\bibitem{Ch20} Choulli, New global logarithmic stability result for the Cauchy problem for elliptic equations. Bull. Aust. Math. Soc. 101 (1) (2020) 141-145.  

\bibitem{Ch19}Choulli, M. :  Inverse problems for Schrödinger equations with unbounded potentials. Notes of the course given during AIP 2019 summer school. arXiv:1909.11133.

\bibitem{Ch} M. Choulli, M. : Une introduction aux probl\`emes inverses elliptiques et paraboliques (An introduction to elliptic and parabolic inverse problems). Math. Appl. (Berlin), 65 [Mathematics and Applications]
Springer-Verlag, Berlin, 2009, xxii+249 pp.

\bibitem{Is} Isakov, V. :  Inverse problems for partial differential equations. Appl. Math. Sci., 127. Springer-Verlag, New York, 1998, xii+284 pp.
 
\bibitem{JK} Jerison, D. ; Kenig, C. E. : Unique continuation and absence of positive eigenvalues for Schrödinger operators. Ann. of Math. (2) 121 (1985), no. 3, 463-494.
 
\bibitem{Po}  Pohjola, V. : Multidimensional Borg-Levinson theorems for unbounded potentials. Asymptot. Anal. 110 (3-4) (2018),  203-226.

\end{thebibliography}
\end{document}